\documentclass[11pt,letterpaper]{amsart}
\usepackage{tikz}
\usetikzlibrary{intersections, calc, arrows.meta}
\usepackage[utf8]{inputenc}
\usepackage{amsmath} 
\usepackage{amsmath}
\usepackage{amssymb} 
\usepackage{amsfonts} 
 \usepackage{amsthm}
 \usepackage{mathtools}
 \usepackage{latexsym}
 \usepackage{amscd}
 \usepackage[all]{xy}
 \usepackage{mathrsfs}
 \usepackage{yhmath}
\newtheorem{dfn}{Definition}[section]
 \newtheorem{them}[dfn]{Theorem}
 \newtheorem{lem}[dfn]{Lemma}
 \newtheorem{prp}[dfn]{Proposition}

 \newtheorem{cla}[dfn]{Claim}
 \newtheorem{Que}[dfn]{Question}
 \newtheorem{Cor}[dfn]{Corollary}
  \theoremstyle{definition}
\newtheorem{rem}[dfn]{Remark}

 \usepackage{scalerel,stackengine}
\stackMath
\newcommand\reallywidecheck[1]{%
\savestack{\tmpbox}{\stretchto{%
  \scaleto{%
    \scalerel*[\widthof{\ensuremath{#1}}]{\kern-.6pt\bigwedge\kern-.6pt}%
    {\rule[-\textheight/2]{1ex}{\textheight}}
  }{\textheight}%
}{0.5ex}}%
\stackon[1pt]{#1}{\scalebox{-1}{\tmpbox}}%
}
\begin{document}

\pagestyle{plain}
\thispagestyle{plain}

\title[]{Dynamically convex and global surface of section in $L(p,p-1)$ from the  viewpoint of ECH}
\author[Taisuke SHIBATA]{Taisuke SHIBATA}
\address{Research Institute for Mathematical Sciences, Kyoto University, Kyoto 606-8502,
JAPAN.}
\email{shibata@kurims.kyoto-u.ac.jp}

\date{\today}

\begin{abstract}
The notion of dynamically convex has been studied since it was introduced by Hofer, Wysocki and Zehnder. In particular, they showed that there muse exist a global surface of section of disk type binding a periodic orbit for the Reeb vector field in dynamically convex $(S^{3},\lambda)$ by using pseudoholomorphic curves. Recently Hryniewicz and Salomão showed the same result for $L(2,1)$ by developing the original technique and after that  Schneider generalized it to $(L(p,1),\xi_{\mathrm{std}})$. The main purpose of this paper is to introduce an alternative approach of using Embedded contact homology. In particular, we find a global surface of section of disk type in dynamically convex $(L(p,p-1),\xi_{\mathrm{std}})$ and relate their periods to the first ECH spectrum. 
\end{abstract}

\maketitle

\tableofcontents
\section{Introduction and main results}

\subsection{Introduction}
A closed contact three manifold $(Y,\lambda)$ is a pair of a closed contact three manifold $Y$ with a contact 1-form $\lambda$. 
A  contact form $\lambda$ on $Y$ defines the Reeb vector field $X_{\lambda}$ and the contact structure $\xi=\mathrm{Ker}\lambda$.  A periodic orbit is a map $\gamma:\mathbb{R}/T_{\gamma}\mathbb{Z}\to Y$ satisfying $\Dot{\gamma}=X_{\lambda}\circ \gamma$ for some $T_{\gamma}>0$ and we write $\gamma^{p}$ for $p\in \mathbb{Z}$ as a periodic orbit of composing $\gamma$ with 
the natural projection $\mathbb{R}/pT_{\gamma}\mathbb{Z}\to \mathbb{R}/T_{\gamma}\mathbb{Z}$. In this paper, we consider two orbit to be equivalent if they are  equivalent as currents. A periodic orbit $\gamma$ is non-degenerate if the return map $d\phi^{T_{\gamma}}|_{\xi}:\xi_{\gamma(0)}\to\xi_{\gamma(0)}$ has no eigenvalue $1$ and simple if $\gamma$ is an embedding map where $\phi^{t}$ is the flow of $X_{\lambda}$. We call $(Y,\lambda)$ non-degenerate if all periodic orbits are non-degenerate.
According to \cite{HWZ2}, for a (possibly degenerate) orbit $\gamma$ and a symplectic trivialization $\tau:\gamma^{*}\xi \to \mathbb{R}/T_{\gamma}\mathbb{Z} \times \mathbb{C}$, Conley-Zehnder index is defined and we write $\mu_{\tau}(\gamma)$. We would recall Conley-Zehnder index in the next section.  

In this paper we focus on lens spaces $L(p,q)$ and the standard contact structure $\xi_{\mathrm{std}}$ on them. Here $p\geq q>0$  are mutually prime and the standard contact structure $\xi_{\mathrm{std}}$ is defined as follows.
Consider a contact 3-sphere $(\partial B(1),\lambda_{0}|_{\partial B(1)})$ where $\partial B(1)=\{(z_{1},z_{2})\in \mathbb{C}^{2}||z_{1}|^{2}+|z_{2}|^{2}=1\}$, $\lambda_{0}=\frac{i}{2}\sum_{i=1,2}(z_{i}d\Bar{z_{i}}-\Bar{z_{i}}dz_{i})$. The action $(z_{1},z_{2})\mapsto (e^{\frac{2\pi i}{p}}z_{1},e^{\frac{2\pi iq}{p}}z_{2})$ preserves $(\partial B(1),\lambda_{0}|_{\partial B(1)})$ and the tight contact structure. Hence we have the quotient space which is a contact manifold and write  $(L(p,q),\lambda_{p,q})$, $\xi_{\mathrm{std}}=\mathrm{Ker}\lambda_{p,q}$.

\begin{dfn}
    $(S^{3},\lambda)$ is strictly convex if there is an embedding $i:S^{3}\to \mathbb{C}^{2}$ such that $i(S^{3})\subset \mathbb{C}^{2}$ is a strictly convex hypersurface surrounding $0$ and $i^{*}\lambda_{0}=\lambda$. More generally, consider $(L(p,q),\lambda)$ and  the covering map  $\rho:S^{3}\to L(p,q)$.
    $(L(p,q),\lambda)$ is strictly convex if there is an embedding $i:S^{3}\to \mathbb{C}^{2}$ with $i^{*}\lambda_{0}=\rho^{*}\lambda$ such that  $i(S^{3})\subset \mathbb{C}^{2}$ is a strictly convex hypersurface surrounding $0$ and in addition $i$ is an equivalent map under $(z_{1},z_{2})\mapsto (e^{\frac{2\pi i}{p}}z_{1},e^{\frac{2\pi iq}{p}}z_{2})$.
\end{dfn}
\begin{dfn}[{\cite{HWZ4}}]
    Assume that a contact three manifold $(Y,\lambda)$ satisfies $c_{1}(\xi)|_{\pi_{2}(Y)}=0$. $\lambda$ is called dynamically convex if $\mu_{\mathrm{disk}}(\gamma)\geq3$ for any contractible periodic orbit $\gamma$, where $\mu_{\mathrm{disk}}(\gamma)$ is defined as follows. Take  a smooth map $u:\{|z|\leq 1|z\in \mathbb{C}\}:=\mathbb{D} \to Y$ with $u(e^{2\pi i t})=\gamma(T_{\gamma}t)$ and a global trivialization $v:u^{*}\xi \to \mathbb{D}\times \mathbb{C}$. Then $\mu_{\mathrm{disk}}(\gamma):=\mu_{v|_{\partial \mathbb{D}}}(\gamma)$. Note that $\mu_{\mathrm{disk}}(\gamma)$ is independent of the choice of $u$ since $c_{1}(\xi)|_{\pi_{2}(Y)}=0$.
\end{dfn}
\begin{them}\cite{HWZ4}
    If  $(S^{3},\lambda)$ (resp. $(L(p,q),\lambda)$) is strictly convex, then  $(S^{3},\lambda)$  (resp. $(L(p,q),\lambda)$) is dynamically convex.
\end{them}

\begin{rem}
        It follows from the definition that the condition of dynamical convexity is preserved under taking a finite cover. 
\end{rem}

The notion of dynamically convex is a generalization of strictly convex energy hypersurfaces in $\mathbb{R}^{4}$ from the viewpoint of dynamical system and has been studied since it was introduced in \cite{HWZ4}. To explain the results, we start with introducing some notions.

\begin{dfn}
    A knot $K\subset Y$ is called $p$-unknotted if there exists an immersion $u:\mathbb{D}\to Y$ such that $u|_{\mathrm{int}(\mathbb{D})}$ is embedded and $u|_{\partial \mathbb{D}}:\partial \mathbb{D}\to K$ is a $p$-covering map.
    \end{dfn}

    \begin{dfn}\cite[cf. Subsection 1.1]{BE}
    Assume that a knot $K\in Y$ is $p$-unknotted, transversal to $\xi$ and oriented by the co-orientation of $\xi$. Let $u:\mathbb{D}\to Y$ be  an immersion  such that $u|_{\mathrm{int}(\mathbb{D})}$ is embedded and $u|_{\partial \mathbb{D}}:\partial \mathbb{D}\to K$ is a $p$-covering map. Take a non-vanishing section $Z:\mathbb{D}\to u^{*}\xi$ and consider the immersion $\gamma_{\epsilon}:t\in \mathbb{R}/\mathbb{Z} \to \mathrm{exp}_{u(e^{2\pi i t})}(\epsilon Z(u(e^{2\pi i t})))\in Y\backslash K$ for small $\epsilon>0$. 
    
    Define the rational self-linking number $\mathrm{sl}(K,u)\in \mathbb{Q}$ as
    \begin{equation*}
        \mathrm{sl}(K,u)=\frac{1}{p^{2}} (\mathrm{algebraic\,\,intersection\,\,number\,\,of}\,\, \gamma_{\epsilon}\,\,\mathrm{with}\,\,u)
    \end{equation*}
    If  $c_{1}(\xi)|_{\pi_{2}(Y)}=0$, $\mathrm{sl}(K,u)$ is independent of $u$. Hence  we write $\mathrm{sl}(K)$.
\end{dfn}
\begin{rem}
In generall, (rational) self-linking number is defined for rationally null-homologous knot by using a (rational) Seifert surface. See \cite{BE}.
\end{rem}

Note that we assume that lens spaces $L(p,q)$ contain $S^{3}$ as a lens space with $p=1$. 
\begin{dfn}
    Let $(Y,\lambda)$ be a contact three-manifold.  A global surface of section for $X_{\lambda}$ on a 3-manifold M
is a compact immersed surface $\Sigma \to Y$ such that
\item[(1).]$\Sigma\backslash \partial \Sigma$ is embedded,
\item[(2).] $X_{\lambda}$ is transversal to $\Sigma\backslash \partial \Sigma$,
\item[(3).] $\partial \Sigma$ consists of periodic orbits of $X_{\lambda}$,
\item[(4).] For every $x\in Y\backslash \partial \Sigma$, there are $-\infty<t_{x}^{-}<0<t_{x}^{+}<+\infty$ such that $\phi^{t^{\pm}_{x}}(x)\in \Sigma$ where $\phi^{t}$ is the flow of $X_{\lambda}$.
\end{dfn}

According to  \cite{HWZ3,HWZ4,HWZ5,Hr1},  if  a dynamically convex contact 3-manifold $(Y,\lambda)$ with $\mathrm{Ker}\lambda=\xi$ has a 1-unknotted simple orbit $\gamma$ with  self-linking number -1, then $(Y,\xi)$ is contactomorphic to $(S^{3},\xi_{\mathrm{std}})$ and in addition there is a global surface of section of disk type binding $\gamma$. Moreover  any dynamically convex contact form $\lambda$ on $S^{3}$ has a 1-unknotted simple orbit $\gamma$ with  self-linking number $-1$ and $\mu_{\mathrm{disk}}(\gamma)=3$.

In \cite{HrLS}, Hryniewicz,  Licata, and  Salomão generalized the result to lens spaces as follows. Let $p\in \mathbb{Z}_{>0}$.
A closed connected contact 3-manifold $(Y,\xi)$ is contactomorphic to $(L(p,k),\xi_{\mathrm{std}})$ for some $k$ if and only if there is a dynamically convex contact form $\lambda$ with $\mathrm{Ker}\lambda=\xi$ 
such that $X_{\lambda}$ has  a $p$-unknotted simple orbit $\gamma$ with self-linking number $-\frac{1}{p}$. In connection with these results,  Hryniewicz and  Salomão showed the following result.

\begin{them}\label{fundament}\cite[Theorem 1.7, Corollary 1.8]{HrS}
If $\lambda$ is any dynamically convex contact form on $L(p,q)$, then for every
$p$-unknotted simple orbit $\gamma$ with $\mathrm{sl}(\gamma)=-\frac{1}{p}$, $\gamma^{p}$ must bound a 
a disk which is a global surface of section for the Reeb flow. Moreover, this disk is a page of a rational open book decomposition of $L(p,q)$ with binding $\gamma$ such that all pages are disk-like global surfaces of section.  
\end{them}

It is natural to consider the next question.
\begin{Que}
     Does  any dynamically convex contact form $\lambda$ on $L(p,q)$  with $\mathrm{Ker}\lambda=\xi_{\mathrm{std}}$ have a $p$-unknotted simple orbit with self-linking number $-\frac{1}{p}$ ?
\end{Que}

For $(L(p,q),\lambda)$, define 
\begin{equation*}
    \mathcal{S}_{p}(L(p,q),\lambda):=\{\gamma \mathrm{\,\,simple\,\,orbit\,\,of}\,\,(L(p,q),\lambda)|\,\,p\mathrm{-unknotted},\,\,\mathrm{sl}(\gamma)=-\frac{1}{p}\,\}.
\end{equation*}
We sometimes write $\mathcal{S}_{p}$ instead of $ \mathcal{S}_{p}(L(p,q),\lambda)$ if there is no confusion. 

Originally as mentioned, Hofer, Wysocki and Zehnder proved in \cite{HWZ4} that any dynamically convex contact form $\lambda$ on $S^{3}$ admits $\gamma \in \mathcal{S}_{1}$ with $\mu_{\mathrm{disk}}(\gamma)=3$. In particular, they used the compactification of pseudoholomorphic curves coming from an ellipsoid in a symplectic manifold which connects the symplectization of the ellipsoid with the one of $(S^{3},\lambda)$. After that, Hryniewicz and  Salomão proved in \cite{HrS} that any dynamically convex contact form $\lambda$ on $L(2,1)$ admits $\gamma \in \mathcal{S}_{2}$ with $\mu_{\mathrm{disk}}(\gamma^{2})=3$ especially such  orbits are elliptic. Recently, Schneider generalized in \cite{Sch} the results to $(L(p,1),\lambda)$ with $\mathrm{Ker}\lambda=\xi_{\mathrm{std}}$. 

 In any cases, the original ideas in \cite{HWZ4} of considering the compactification of pseudoholomorphic curves coming from an ellipsoid are essentially used. Note that  in lens spaces, we should consider the multiple covers of pseudoholomorphic curves and in \cite{HrS,Sch} they  observe that such curves behave good in $(L(p,1),\xi_{\mathrm{std}})$.
 
In summary;
\begin{them}\label{original}\cite{HWZ4,HrS,Sch}
    Let $(L(p,1),\lambda)$ be  dynamically convex  with $\mathrm{Ker}\lambda=\xi_{\mathrm{std}}$. Then, there is a simple orbit $\gamma$ such that $\gamma \in \mathcal{S}_{p}$ and $\mu_{\mathrm{disk}}(\gamma^{p})=3$.
\end{them}
 
 The main result of this paper is to find a simple orbit $\gamma \in \mathcal{S}_{p}$ in $(L(p,p-1),\xi_{\mathrm{std}})$ by using Embedded contact homology and relate their periods  with the first ECH spectrum. In particular, $U$-map on ECH plays important roles in the proof.

\subsection{Main results}

According to \cite{HrHuRa}, the next Theorem holds.
\begin{them}\label{hutchings}\cite{HrHuRa}
    Assume that $(S^{3},\lambda)$ is dynamically convex  and  $c_{1}^{\mathrm{ECH}}(S^{3},\lambda)$ is the first ECH spectrum of $(S^{3},\lambda)$ (explained in the next subsection). Then
\begin{equation*}
    c_{1}^{\mathrm{ECH}}(S^{3},\lambda)=\inf_{\gamma\in \mathcal{S}_{1}}\int_{\gamma}\lambda
\end{equation*}
\end{them}
Consider $(L(p,p-1),\lambda)$ with $\mathrm{Ker}\lambda=\xi_{\mathrm{std}}$. Since $\xi_{\mathrm{std}}$ is trivial, we can take a global symplectic trivialization $\tau_{0}:\xi_{\mathrm{std}}\to L(p,p-1)\times \mathbb{C}$. For an orbit $\gamma$, define $\mu(\gamma)$ as the Conley-Zehnder index with respect to the global trivialization. We note that if $\gamma$ is contractible, $\mu(\gamma)=\mu_{\mathrm{disk}}(\gamma)$.

Our main results are as follows.
\begin{them}\label{maintheorem}
\item[(A).] If $(L(2,1),\lambda)$ is strictly convex or non-degenerate dynamically convex, then
\begin{equation*}
    \inf_{\gamma\in \mathcal{S}_{2},\mu(\gamma)=1}\int_{\gamma}\lambda=\frac{1}{2}\,c_{1}^{\mathrm{ECH}}(L(2,1),\lambda).
\end{equation*}
Moreover there exists  $\gamma \in \mathcal{S}_{2}$ satisfying $\mu(\gamma)=1$ and
\begin{equation*}
    \int_{\gamma}\lambda= \frac{1}{2}\,c_{1}^{\mathrm{ECH}}(L(2,1),\lambda).
\end{equation*}
\item[(B).]Suppose that $(L(p,p-1),\lambda)$ is strictly convex or non-degenerate dynamically convex.  If $p=3$ or $4$ or $6$, there exists  $\gamma \in \mathcal{S}_{p}$ satisfying $\mu(\gamma)=1$ and
\begin{equation*}
    \int_{\gamma}\lambda\leq \frac{1}{2}\,c_{1}^{\mathrm{ECH}}(L(p,p-1),\lambda).
\end{equation*}
\item[(C).] Suppose that $(L(p,p-1),\lambda)$ is non-degenerate dynamically convex.. Then for any $p$, there exists  $\gamma \in \mathcal{S}_{p}$ satisfying $\mu(\gamma)=1$ and
\begin{equation*}
 \int_{\gamma}\lambda\leq \frac{1}{2}\,c_{1}^{\mathrm{ECH}}(L(p,p-1),\lambda).
\end{equation*}
\end{them}
\begin{rem}\label{mainrem}
\item[(1).] The reason why we impose the condition of strictly convex in degenerate case on the assumption is that we can take a sequence of non-degenerate dynamically convex contact forms $\lambda_{n}\to \lambda$ for any strictly convex contact form $\lambda$. In general, for any dynamically convex contact form $\lambda$, we can take a sequence of non-degenerate contact forms $\lambda_{n}\to \lambda$ such that $\lambda_{n}$ is dynamically convex up to some period. But  the author thinks that the difficult problem is that Lemma \ref{emptyindex1} and Lemma \ref{existenceofhol} do not work in this sequence.
    
    \item[(2).] In general,  in order to prove something under degeneracy, we usually consider it as a limiting case of non-degenerate objects.  The reason why we only consider $p=3,4,6$ in the Theorem \ref{maincor}  (C) is that otherwise the limit of simple orbits might not be  simple. See Section 5 for more details.
    \item[(3).] In contrast to Theorem \ref{original}, if $p\geq 3$ in $L(p,p-1)$, the orbit $\gamma$ is $\mu(\gamma)=1$ but $\mu(\gamma^{p})$ might not be equal to $3$. The author thinks that this relates to the difficulty of applying the original method in \cite{HWZ4} to $L(p,q)$ under $q\neq1$. Note that  Schneider mentions in \cite[Remark 1.6]{Sch} that even if we can find an orbit $\gamma \in \mathcal{S}_{p}$ in $L(p,q)$ under $q\neq 1$, $\mu(\gamma^{p})$ might not be equal to $3$.
\end{rem}
As an immediate application of Theorem \ref{maintheorem} and Theorem \ref{fundament}, we have 
\begin{Cor}
    If $(L(p,p-1),\lambda)$  is non-degenerate dynamically convex, then $X_{\lambda}$ admits a $p$-unknotted closed Reeb orbit $\gamma$ with $\mu(\gamma)=1$ which is the binding of a rational open book decomposition. Each page of the open book is a rational  global surface of section of disk type whose contact area is not larger than $\frac{p}{2}c_{1}^{\mathrm{ECH}}(L(p,p-1),\lambda)$. Moreover, if $(L(p,p-1),\lambda)$ is strictly convex and $p=2,3,4,6$, we can exclude the condition of non-degeneracy.
\end{Cor}

\begin{Cor}\label{maincor}
Assume that $(L(p,p-1),\lambda)$ is strictly convex or  non-degenerate dynamically convex. Let  $\rho:S^{3}\to L(p,p-1)$ be the covering map. Then, 
\begin{equation*}
     c_{1}^{\mathrm{ECH}}(S^{3},\rho^{*}\lambda) \leq \frac{p}{2}c_{1}^{\mathrm{ECH}}(L(p,p-1),\lambda).
\end{equation*}
\end{Cor}
\begin{proof}[\bf Proof of Corollary \ref{maincor}]
If non-degenerate, it follows   from Theorem \ref{maintheorem} that  there is $\gamma\in \mathcal{S}_{p}(L(p,p-1),\lambda)$ with $\mu(\gamma)=1$ and $\int_{\gamma}\lambda\leq\frac{1}{2} c_{1}^{\mathrm{ECH}}(L(p,p-1),\lambda)$.
Let $\Tilde{\gamma}$ be the periodic orbit in $(S^{3},\rho^{*}\lambda_{n})$ such that $\rho|_{\Tilde{\gamma}}:\Tilde{\gamma}\to \gamma$ is a $p$-fold  cover. It is obvious that $\Tilde{\gamma}\in \mathcal{S}_{1}(S^{3},\rho^{*}\lambda)$. This implies that from Theorem \ref{hutchings},
\begin{equation}
     c_{1}^{\mathrm{ECH}}(S^{3},\rho^{*}\lambda)=\inf_{\gamma\in \mathcal{S}_{1}}\int_{\gamma}\lambda \leq \int_{\Tilde{\gamma}}\rho^{*}\lambda=p\int_{\gamma}\lambda \leq \frac{p}{2} c_{1}^{\mathrm{ECH}}(L(p,p-1),\lambda).
\end{equation}
Therefore, we have $c_{1}^{\mathrm{ECH}}(S^{3},\rho^{*}\lambda) \leq \frac{p}{2} c_{1}^{\mathrm{ECH}}(L(p,p-1),\lambda)$.

If strictly convex, we consider a sequence of strictly convex 
and non-degenerate contact forms $\lambda_{n}$ such that $\mathrm{Ker}\lambda_{n}=\xi_{\mathrm{std}}$ and $\lambda_{n}\to \lambda$ in $C^{\infty}$-topology.  Then it follows from the property of ECH spectrum that  $c_{1}^{\mathrm{ECH}}(S^{3},\rho^{*}\lambda) \leq \frac{p}{2} c_{1}^{\mathrm{ECH}}(L(p,p-1),\lambda)$ as the limit. We complete the proof.
\end{proof}

\begin{rem}
The Weyl law with respact to ECH spectrum proved in \cite{CHR} says  $\lim_{k\to \infty}\frac{c_{k}^{\mathrm{ECH}}(S^{3},\rho^{*}\lambda)^{2}}{k}=2\mathrm{Vol}(S^{3},\rho^{*}\lambda)$ and $\lim_{k\to \infty}\frac{c_{k}^{\mathrm{ECH}}(L(p,p-1),\lambda)^{2}}{k}=2\mathrm{Vol}(L(p,p-1),\lambda)$. This implies that
\begin{equation}
    \lim_{k\to \infty}\frac{c_{k}^{\mathrm{ECH}}(S^{3},\rho^{*}\lambda)}{c_{k}^{\mathrm{ECH}}(L(p,p-1),\lambda)}=\sqrt{\frac{\mathrm{Vol}(S^{3},\rho^{*}\lambda)}{\mathrm{Vol}(L(p,p-1),\lambda)}}=\sqrt{p}.
\end{equation} 
But if $\lambda$ is strictly  convex, it follows  from  Corollary \ref{maincor} that $\frac{c_{1}^{\mathrm{ECH}}(S^{3},\rho^{*}\lambda)}{c_{1}^{\mathrm{ECH}}(L(p,p-1),\lambda)}\leq \frac{p}{2}$. This implies that if $p=2,3$,   the ratio is not close to $\sqrt{p}$ when $k=1$.
\end{rem}

Theorem \ref{maintheorem} (A) partially answers the question \cite[Question 1, Remark 1.9]{Fe}. In connection with this, we  conjecture the following.
\begin{Que}
Suppose that $(L(2,1),\lambda)$ is strictly convex. Is the periodic orbit with minimum period  in  $\mathcal{S}_{2}(L(2,1),\lambda)$?
\end{Que}

\subsection{Idea and outline of this paper}
The basic idea is to find the orbits   from the algenbraic structure of ECH called $U$-map recalled in the next section. From the behaviors of $J$-holomorphic curves and algebraic structure, we can find an ECH generator $\alpha$ satisfying $\langle U_{J,z}\alpha,\emptyset \rangle \neq 0$ and $A(\alpha)\leq c_{1}^{\mathrm{ECH}}(L(p,p-1),\lambda)$. In addition, we can see that if $\langle U_{J_{z}}\alpha,\emptyset \rangle \neq 0$, $\alpha$ is described as either
\begin{itemize}
    \item[(1).] $\alpha=(\gamma,p)$ with $\gamma\in \mathcal{S}_{p}$, $\mu(\gamma^{p})=3$ and $\mu(\gamma)=1$, or
    \item[(2).]$\alpha=(\gamma_{1},1)\cup{(\gamma_{2},1)}$ with $\gamma_{1}, \gamma_{2}\in \mathcal{S}_{p}$, $\mu(\gamma_{1})=\mu(\gamma_{2})=1$.
\end{itemize}

In Section 2, we recall Conley-zehnder index and ECH. In Section 3, we focus on non-degenerate cases and study the behaviors of Conley-Zehnder index and $J$-holomorphic curves. In Section 4, we show the above properties under non-degenerate cases. In particular, we complete the proof for non-degenerate cases. In Section 5, we extend the result of Sectioon 4 to degenerate cases under strictly convex and $p=2,3,4,6$. In Section 6, we prove Theorem \ref{familyofhol} used in Section 3. 
\subsection*{Acknowledgement}
The author would like to thank his advisor Professor Kaoru Ono for his encouragement and support, U. L. Hryniewicz for his encouraging comment and Brayan Ferreira for his pointing out the relation with his work. The author also would like to  thank Michael Hutching for sharing the information of \cite{HrHuRa} and  being so kind to me when the author stayed at Berkeley.
This work was supported by JSPS KAKENHI Grant Number JP21J20300.

\section{Preliminaries}
\subsection{Conley-Zehnder index}
In this subsection, we recall Conley-zehnder index and its properties introduced in \cite{HWZ2}. The contents in this subsection are based on \cite{HWZ2,HWZ4}.

Consider a smooth arc $[0,1]\ni t\mapsto S(t)$ of symmetric matrices. Then we can define a self-adjoint operator $L_{S}$ in $L^{2}(S^{1},\mathbb{C})$ whose domain is $W^{1,2}(S^{1},\mathbb{C})$ by
\begin{equation}
    L_{S}=-i\partial_{t}-S(t).
\end{equation}
Since $W^{1,2}(S^{1},\mathbb{C}) \hookrightarrow L^{2}(S^{1},\mathbb{C})$ is compact, its spectrum $\sigma(L_{S})$ consists of eigenvalues of $L_{S}$ and is in $\mathbb{R}$. In addition, $\sigma(L_{S})$ is a countable set and has no accumulation point other than $\pm \infty$. For $\eta \in \sigma(L_{S})$, take an eigenfunction $e_{\eta}:S^{1}\to \mathbb{C}$. 
Let $w(S,\eta)\in \mathbb{Z}$ denote the winding number of $e_{\eta}$.  Note that $w(S,\eta)$  is independent of $e_{\eta}$. If $\eta_{1}\leq \eta_{2}$, then $w(S,\eta_{1})\leq w(S,\eta_{2})$. Moreover for any $k\in \mathbb{Z}$, there are precisely two eigenvalues (multiplicities counted) with winding number $k$. In particular, this means that $w(S,\cdot):\sigma(L_{S})\to \mathbb{Z}$ is monotone and surjective

 Consider a a smooth path $\varphi:\mathbb{R}\to Sp(1)$ in 2-dimensional symplectic matrices with $\varphi(0)=id$ and $\varphi(t+1)=\varphi(t)\varphi(1)$ for any $t\in \mathbb{R}$. Define a smooth arc $[0,1]\ni t\mapsto S_{\varphi}(t)$ of symmetric matrices by $S_{\varphi}(t)=-i\dot{\varphi}(t)\varphi(t)$.

\begin{dfn}
    For $\varphi$ as above,  let $\eta^{<}:=\mathrm{max}\{ \sigma(L_{S_{\phi}})\cap{(-\infty,0)}\}$ and $\eta^{\geq}:=\mathrm{min}\{ \sigma(L_{S_{\varphi}})\cap{[0,+\infty)}\}$. The Conley-Zehnder index of $\varphi$ is defined as
    \begin{equation}
    \mu_{CZ}(\phi):=w(S_{\varphi},\eta^{<})+w(S_{\varphi},\eta^{\geq}).
    \end{equation}
\end{dfn}
We note that it is obvious that $ \mu_{CZ}$ is lower semi-continuous.

For $k\in \mathbb{Z}_{>0}$, define $\rho_{k}:\mathbb{R} \to \mathbb{R}$ as $t \mapsto kt$. The next proposition is used in Section 5.
\begin{prp}\label{conleycovering}
    Consider a smooth path $\varphi:\mathbb{R}\to Sp(1)$ in  symplectic matrices with $\varphi(0)=id$ and $\varphi(t+1)=\varphi(t)\varphi(1)$ for any $t\in \mathbb{R}$.
    \item[(1).] If $\mu_{CZ}(\varphi)=2n$ for $n\in \mathbb{Z}$, then $\mu_{CZ}(\varphi\circ \rho_{k})=2kn$ for every $k\in \mathbb{Z}_{>0}$.
       \item[(2).] If $\mu_{CZ}(\varphi)\geq3$, then $\mu_{CZ}(\varphi\circ \rho_{k})\geq 2k+1$ for every $k\in \mathbb{Z}_{>0}$.
\end{prp}
\begin{proof}[\bf Proof of Proposition \ref{conleycovering}]
    Let $e_{\eta}:S^{1}\to \mathbb{C}$ be an eigenfunction of $L_{S_{\varphi}}$ with eigenvalue $\eta \in \sigma(L_{S_{\varphi}})$. Then it is obvious that $L_{S_{\varphi\circ \rho_{k}}}e_{\eta}(kt)=\eta e_{\eta}(kt)$. Therefore $\sigma(L_{S_{\varphi}})\subset \sigma(L_{S_{\varphi\circ \rho_{k}}})$ and $w(S_{\varphi \circ \rho_{k}},\eta)=kw(S_{\varphi},\eta)$ for every $\eta \in \sigma(L_{S_{\varphi}})$. 
    In addition, we note that $w(S_{\phi},\eta^{<})=w(S_{\varphi},\eta^{\geq})$ or $w(S_{\phi},\eta^{<})+1=w(S_{\varphi},\eta^{\geq})$ because the map $w(S_{\varphi},\cdot):\sigma(L_{S_{\varphi}})\to \mathbb{Z}$ is monotone and surjective.

        Supppose that $\mu_{CZ}(\varphi)=2n$ for $n\in \mathbb{Z}$. Since $\mu_{CZ}(\varphi)=w(S_{\varphi},\eta^{<})+w(S_{\varphi},\eta^{\geq})$ is even, we have $w(S_{\varphi},\eta^{<})=w(S_{\varphi},\eta^{\geq})=n$. Fix $k\in \mathbb{Z}_{>0}$. Then we have $\eta^{<}:=\mathrm{max}\{ \sigma(L_{S_{\varphi\circ \rho_{k}}})\cap{(-\infty,0)}\}=\mathrm{max}\{ \sigma(L_{S_{\varphi}})\cap{(-\infty,0)}\}$ and $\eta^{\geq}:=\mathrm{min}\{ \sigma(L_{S_{\varphi\circ \rho_{k}}})\cap{[0,+\infty)}\}=\mathrm{min}\{ \sigma(L_{S_{\varphi}})\cap{[0,+\infty)}\}$. This follows easily from  $w(S_{\varphi\circ \rho_{k}},\eta^{<})=w(S_{\varphi\circ \rho_{k}},\eta^{\geq})=kn$ and the monotonicity of $w(S_{\varphi \circ \rho_{k}},\cdot)$. Therefore we have $\mu_{CZ}(\varphi\circ \rho_{k})=2kn$.

For the proof of (2), see \cite[Theorem 3.6]{HWZ4}.
\end{proof}

Consider a periodic orbit $\gamma:\mathbb{R}/T_{\gamma}\mathbb{Z}\to Y$ of $(Y,\lambda)$ and  a symplectic trivialization $\tau:\gamma^{*}\xi \to \mathbb{R}/T_{\gamma}\mathbb{Z} \times \mathbb{C}$. Then we have a symplectic path $ \mathbb{R}\ni t\mapsto \phi_{\gamma,\tau}(t):=\tau(\gamma(T_{\gamma}t)) \circ d\phi^{tT_{\gamma}}|_{\xi}\circ \tau^{-1}(\gamma(0))$ which satisfies $\phi_{\gamma,\tau}(t+1)=\phi_{\gamma,\tau}(t)\phi_{\gamma,\tau}(0) $ for any $t\in \mathbb{R}$.

Now, we define the Conley-Zender index of $\gamma$ with respect to a trivialization $\tau$ as
\begin{equation}
    \mu_{\tau}(\gamma):=\mu_{CZ}(\phi_{\gamma,\tau}).
\end{equation}
Note that  $ \mu_{\tau}$ is independent of the choice of a trivialization in the same homotopy class of $\tau$.

Consider a periodic orbit $\gamma$. If the eigenvalues of the return map $d\phi^{T_{\gamma}}|_{\xi}:\xi_{\gamma(0)}\to\xi_{\gamma(0)}$ are positive (resp. negative) real, $\gamma$ is called positive (resp. negative) hyperbolic. If the eigenvalues of the return map $d\phi^{T_{\gamma}}|_{\xi}:\xi_{\gamma(0)}\to\xi_{\gamma(0)}$ are on theunit circle in $\mathbb{C}$, $\gamma$ is called elliptic.

If $\gamma^{n}$ is non-degenerate for every $n \in \mathbb{Z}_{>0}$, it is well-known that the Conley-Zehnder index behave good as follows. 

\begin{prp}
    Let $\gamma$ be a orbit such that $\gamma^{n}$ is non-degenerate for every $n\in \mathbb{Z}_{>0}$.  Fix a trivialization $\tau$ of the contact plane over $\gamma$. Consider the Conley-Zehnder indices of the multiple covers with respect to $\tau$. Write $
    \mu_{\tau}(\gamma^{n}):=\mu_{CZ}(\phi_{\gamma,\tau}\circ \rho_{n})$.
    \item[(1).] If $\gamma$ is hyperbolic, $\mu_{\tau}(\gamma^{n})=n\mu_{\tau}(\gamma)$ for every $n\in \mathbb{Z}_{>0}$.
    \item[(2).] If $\gamma$ is elliptic, there is $\theta \in \mathbb{R}\backslash \mathbb{Q}$ such that $\mu_{\tau}(\gamma^{n})=2 \lfloor n\theta \rfloor +1$ for every $n \in \mathbb{Z}_{>0}$.

    We call $\theta$ the monodromy angle of $\gamma$.
\end{prp}

For more properties of the Conley-Zehnder index, see \cite{HWZ2,HWZ4}.

\subsection{The definitions and properties of Embedded contact homology}
 In this subsection, we recall the basic construction and properties of ECH.
 
Let $(Y,\lambda)$ be a non-degenerate contact three manifold. For $\Gamma \in H_{1}(Y;\mathbb{Z})$, Embedded contact homology $\mathrm{ECH}(Y,\lambda,\Gamma)$ is defined. At first, we define the chain complex $(\mathrm{ECC}(Y,\lambda,\Gamma),\partial)$. In this paper, we consider ECH over $\mathbb{Z}/2\mathbb{Z}=\mathbb{F}$.

\begin{dfn} [{\cite[Definition 1.1]{H1}}]\label{qdef}
An orbit set $\alpha=\{(\alpha_{i},m_{i})\}$ is a finite pair of distinct simple periodic orbit $\alpha_{i}$ with positive integer $m_{i}$. $\alpha=\{(\alpha_{i},m_{i})\}$ is called an ECH generator If $m_{i}=1$ whenever $\alpha_{i}$ is hyperbolic orbit.
\end{dfn}
Set $[\alpha]=\sum m_{i}[\alpha_{i}] \in H_{1}(Y)$. For two orbit sets $\alpha=\{(\alpha_{i},m_{i})\}$ and $\beta=\{(\beta_{j},n_{j})\}$ with $[\alpha]=[\beta]$, we define  $H_{2}(Y,\alpha,\beta)$ to be the set of relative homology classes of
2-chains $Z$ in $Y$ with $\partial Z =\sum_{i}m_{i} \alpha_{i}-\sum_{j}m_{j}\beta_{j}$ . This is an affine space over $H_{2}(Y)$. 
From now on, we fix the trivialization of the contact surface $\xi$ on each simple orbit and write it as $\tau$. Moreover we assume that the trivialization on any orbit is defined by the trivialization induced by $\tau$ and use the same symbol $\tau$.

\begin{dfn}[{\cite[{\S}8.2]{H1}}]
Let $\alpha_{1}$, $\beta_{1}$, $\alpha_{2}$ and $\beta_{2}$ be orbit sets with $[\alpha_{1}]=[\beta_{1}]$ and $[\alpha_{2}]=[\beta_{2}]$. For a trivialization $\tau$, we can define
\begin{equation}
    Q_{\tau}:H_{2}(Y;\alpha_{1},\beta_{1}) \times H_{2}(Y;\alpha_{2},\beta_{2}) \to \mathbb{Z}
\end{equation}
This is well-defined. Moreover if $Z_{1}\in H_{2}(Y;\alpha_{1},\beta_{1})$ and $Z_{2} \in H_{2}(Y;\alpha_{2},\beta_{2})$, then
\begin{equation}
    Q_{\tau}(Z_{1}+Z_{2},Z_{1}+Z_{2})=Q_{\tau}(Z_{1},Z_{1})+2Q_{\tau}(Z_{1},Z_{2})+Q_{\tau}(Z_{2},Z_{2}).
\end{equation}
See {\cite[{\S}8.2]{H1}} for more details.
\end{dfn}

\begin{dfn}[{\cite[Definition 1.5]{H1}}]
For $Z\in H_{2}(Y,\alpha,\beta)$, we define
\begin{equation}
    I(\alpha,\beta,Z):=c_{1}(\xi|_{Z},\tau)+Q_{\tau}(Z)+\sum_{i}\sum_{k=1}^{m_{i}}\mu_{\tau}(\alpha_{i}^{k})-\sum_{j}\sum_{k=1}^{n_{j}}\mu_{\tau}(\beta_{j}^{k}).
\end{equation}
We call $I(\alpha,\beta,Z)$ an ECH index. Here,  $\mu_{\tau}$ is the Conely Zhender index with respect to $\tau$ and $c_{1}(\xi|_{Z},\tau)$ is a reative Chern number  and $Q_{\tau}(Z)=Q_{\tau}(Z,Z)$. Moreover this is independent of $\tau$ (see  \cite{H1} for more details).
\end{dfn}
Note that ECH index has the property of additivity, that is, for any three orbit sets $\alpha_{1}, \alpha_{2},\alpha_{3}$ and $Z_{1}\in H_{2}(Y;\alpha_{1},\alpha_{2}), Z_{2}\in H_{2}(Y;\alpha_{2},\alpha_{3})$ , we have $I(\alpha_{1},\alpha_{2},Z_{1})+I(\alpha_{2},\alpha_{3},Z_{2})=I(\alpha_{1},\alpha_{3},Z_{1}+Z_{2})$.

For $\Gamma \in H_{1}(Y)$, we define $\mathrm{ECC}(Y,\lambda,\Gamma)$ as freely generated module over $\mathbb{F}$ by ECH generators $\alpha$ such that $[\alpha]=\Gamma$. That is
\begin{equation*}
    \mathrm{ECC}(Y,\lambda,\Gamma):= \bigoplus_{\alpha:\mathrm{ECH\,\,generator\,\,with\,\,}{[\alpha]=\Gamma}}\mathbb{F}\cdot \alpha.
\end{equation*}
To define the differential $\partial:\mathrm{ECC}(Y,\lambda,\Gamma)\to \mathrm{ECC}(Y,\lambda,\Gamma) $, we Fix a generic almost complex structure $J$  on $\mathbb{R}\times Y$ which satisfies $\mathbb{R}$-invariant, $J(\frac{d}{ds})=X_{\lambda}$, $J\xi=\xi$ and $d\lambda(\cdot,J\cdot)>0$.

We consider $J$-holomorphic curves  $u:(\Sigma,j)\to (\mathbb{R}\times Y,J)$ where
the domain $(\Sigma, j)$ is a punctured compact Riemann surface. Here the domain $\Sigma$ is
not necessarily connected.  Let $\gamma$ be a (not necessarily simple) Reeb orbit.  If a puncture
of $u$ is asymptotic to $\mathbb{R}\times \gamma$ as $s\to \infty$, we call it a positive end of $u$ at $\gamma$ and if a puncture of $u$ is asymptotic to $\mathbb{R}\times \gamma$ as $s\to -\infty$, we call it a negative end of $u$ at $\gamma$ ( see \cite{H1} for more details ).

 Let $\alpha=\{(\alpha_{i},m_{i})\}$ and $\beta=\{(\beta_{i},n_{i})\}$ be orbit sets. Let $\mathcal{M}^{J}(\alpha,\beta)$ denote the set of  $J$-holomorphic curves with positive ends
at covers of $\alpha_{i}$ with total covering multiplicity $m_{i}$, negative ends at covers of $\beta_{j}$
with total covering multiplicity $n_{j}$, and no other punctures. Moreover, in $\mathcal{M}^{J}(\alpha,\beta)$, we consider two
$J$-holomorphic curves   to be equivalent if they represent the same current in $\mathbb{R}\times Y$. We sometimes consider an element in $\mathcal{M}^{J}(\alpha,\beta)$ as the image in $\mathbb{R}\times Y$.
 
For $u \in \mathcal{M}^{J}(\alpha,\beta)$, we naturally have $[u]\in H_{2}(Y;\alpha,\beta)$ and set $I(u)=I(\alpha,\beta,[u])$. Moreover we define
\begin{equation}
     \mathcal{M}_{k}^{J}(\alpha,\beta):=\{\,u\in  \mathcal{M}^{J}(\alpha,\beta)\,|\,I(u)=k\,\,\}
\end{equation}

Under this notations, we define $\partial_{J}:\mathrm{ECC}(Y,\lambda,\Gamma)\to \mathrm{ECC}(Y,\lambda,\Gamma)$ as follows.

For admissible orbit set $\alpha$ with $[\alpha]=\Gamma$, we define

\begin{equation}
    \partial_{J} \alpha =\sum_{\beta:\mathrm{\,\,ECH\,\,generator\,\,with\,\,}[\beta]=\Gamma} \# (\mathcal{M}_{1}^{J}(\alpha,\beta)/\mathbb{R})\cdot \beta.
\end{equation}
Note that the above counting is well-defined and $\partial_{J} \circ \partial_{J}=0$ (see \cite{H1,HT1,HT2}, Proposition \ref{indexproperties} ). Moreover, the homology defined by $\partial_{J}$ does not depend on $J$ and if $\mathrm{Ker}\lambda=\mathrm{Ker}\lambda'$ for non-degenerate $\lambda,\lambda'$, there is a natural isomorphism between $\mathrm{ECH}(Y,\lambda,\Gamma)$ and $\mathrm{ECH}(Y,\lambda',\Gamma)$.  Indeed, It is proved in \cite{T1} that there is a natural isomorphism between  ECH and a version of Monopole Floer homology defined in \cite{KM}.

Next, we recall (Fredholm) index.
For $u\in \mathcal{M}^{J}(\alpha,\beta)$, the  its (Fredholm) index is defined by
\begin{equation}
    \mathrm{ind}(u):=-\chi(u)+2c_{1}(\xi|_{[u]},\tau)+\sum_{k}\mu_{\tau}(\gamma_{k}^{+})-\sum_{l}\mu_{\tau}(\gamma_{l}^{-}).
\end{equation}
Here $\{\gamma_{k}^{+}\}$ is the set consisting of (not necessarilly simple) all positive ends of $u$ and $\{\gamma_{l}^{-}\}$ is that one of all negative ends.  Note that for generic $J$, if $u$ is connected and somewhere injective, then the moduli space of $J$-holomorphic
curves near $u$ is a manifold of dimension $\mathrm{ind}(u)$ (see \cite[Definition 1.3]{HT1}).

\subsubsection{$U$-map}
Let $Y$ be connected.
Then there is degree$-2$ map $U$.
\begin{equation*}\label{Umap}
    U:\mathrm{ECH}(Y,\lambda,\Gamma) \to \mathrm{ECH}(Y,\lambda,\Gamma).
\end{equation*}

To define this, choose a base point $z\in Y$ which is especially  not on the image of any Reeb orbits and let $J$ be generic.
Then define a map 
\begin{equation*}
     U_{J,z}:\mathrm{ECC}(Y,\lambda,\Gamma) \to \mathrm{ECC}(Y,\lambda,\Gamma).
\end{equation*}
by
\begin{equation*}
    U_{J,z} \alpha =\sum_{\beta:\mathrm{\,\,ECH\,\,generator\,\,with\,\,}[\beta]=\Gamma} \# \{\,u\in \mathcal{M}_{2}^{J}(\alpha,\beta)/\mathbb{R})\,|\,(0,z)\in u\,\}\cdot  \beta.
\end{equation*}
The above map $U_{J,z}$ commute with $\partial_{J}$ and we can define the $U$ map
as the induced map on homology.  Under the assumption, this map is independent on $z$ (for a generic $J$). See \cite[{\S}2.5]{HT3} for more details. Moreover, in the same reason as $\partial$, $U_{J,z}$ does not depend on $J$ (see \cite{T1}).
\subsubsection{Partition conditions of elliptic orbits}
For $\theta\in \mathbb{R}\backslash \mathbb{Q}$, we define $S_{\theta}$ to be the set of positive integers $q$ such that $\frac{\lceil q\theta \rceil}{q}< \frac{\lceil q'\theta \rceil}{q'}$ for all $q'\in \{1,\,\,2,...,\,\,q-1\}$ and write $S_{\theta}=\{q_{0}=1,\,\,q_{1},\,\,q_{2},\,\,q_{3},...\}$ in increasing order. Also $S_{-\theta}=\{p_{0}=1,\,\,p_{1},\,\,p_{2},\,\,p_{3},...\}$.

\begin{dfn}[{\cite[Definition 7.1]{HT1}}, or {\cite[{\S}4]{H1}}]
For non negative integer $M$, we inductively define the incoming partition $P_{\theta}^{\mathrm{in}}(M)$ as follows.

For $M=0$, $P_{\theta}^{\mathrm{in}}(0)=\emptyset$ and for $M>0$,
\begin{equation*}
    P_{\theta}^{\mathrm{in}}(M):=P_{\theta}^{\mathrm{in}}(M-a)\cup{(a)}
\end{equation*}

where $a:=\mathrm{max}(S_{\theta}\cap{\{1,\,\,2,...,\,\,M\}})$.
 Define outgoing partition
 \begin{equation*}
      P_{\theta}^{\mathrm{out}}(M):= P_{-\theta}^{\mathrm{in}}(M).
 \end{equation*}
 
The standard ordering convention for $P_{\theta}^{\mathrm{in}}(M)$ or $P_{\theta}^{\mathrm{out}}(M)$ is to list the entries
in``nonincreasing'' order.
\end{dfn}

 Let $\alpha=\{(\alpha_{i},m_{i})\}$ and $\beta=\{(\beta_{i},n_{i})\}$. For  $u\in \mathcal{M}^{J}(\alpha,\beta)$, it can be uniquely
written as $u=u_{0}\cup{u_{1}}$ where $u_{0}$ are unions of all components which maps to $\mathbb{R}$-invariant cylinders in $u$ and $u_{1}$ is the rest of $u$.

\begin{prp}[{\cite[Proposition 7.15]{HT1}}]\label{indexproperties}
Suppose that $J$ is generic and $u=u_{0}\cup{u_{1}}\in \mathcal{M}^{J}(\alpha,\beta)$. Then
    \item[(1).] $I(u)\geq 0$
    \item[(2).] If $I(u)=0$, then $u_{1}=\emptyset$
    \item[(3).] If $I(u)=1$, then  $\mathrm{ind}(u_{1})=1$. Moreover  $u_{1}$ is embedded and does not intersect $u_{0}$.
    \item[(4).] If $I(u)=2$ and $\alpha$ and $\beta$ are ECH generators,  then $\mathrm{ind}(u_{1})=2$ .  Moreover  $u_{1}$ is embedded and does not intersect $u_{0}$.
\end{prp}
\begin{prp}[{\cite[Proposition 7.14, 7.15]{HT1}}]\label{partitioncondition} 
Let $\alpha=\{(\alpha_{i},m_{i})\}$ and $\beta=\{(\beta_{j},n_{j})\}$ be ECH generators. Suppose that $I(u)=1$ or $2$ for $u=u_{0}\cup{u_{1}}\in \mathcal{M}^{J}(\alpha,\beta)$. Define $P_{\alpha_{i}}^{+}$ by the set consisting of the multiplicities of the positive ends of $u_{1}$ at covers of $\alpha_{i}$. In the same way,  define $P_{\beta_{j}}^{-}$ for the negative end.
Suppose that $\alpha_{i}$ in $\alpha$ (resp. $\beta_{j}$ in $\beta$) is elliptic orbit with the monodromy angle $\theta_{\alpha_{i}}$ (resp. $\theta_{\beta_{j}}$). Then under the standard ordering convention, $P_{\alpha_{i}}^{+}$ (resp. $P_{\beta_{j}}^{-}$) is an initial segment of $P_{\theta_{\alpha_{i}}}^{\mathrm{out}}(m_{i})$ (resp. $P_{\theta_{\beta_{j}}}^{\mathrm{in}}(n_{j})$).
\end{prp}
\begin{rem}
    There are partition conditions with respect to hyperbolic orbits. But we don't use it in this paper.
\end{rem}

\subsubsection{$J_{0}$ index and topological complexity of $J$-holomorphic curve}

In this subsection, we recall the $J_{0}$ index.

\begin{dfn}[{\cite[{\S}3.3]{HT3}}]
Let $\alpha=\{(\alpha_{i},m_{i})\}$ and $\beta=\{(\beta_{j},n_{j})\}$ be orbit sets with $[\alpha]=[\beta]$.
For $Z\in H_{2}(Y,\alpha,\beta)$, we define
\begin{equation}
    J_{0}(\alpha,\beta,Z):=-c_{1}(\xi|_{Z},\tau)+Q_{\tau}(Z)+\sum_{i}\sum_{k=1}^{m_{i}-1}\mu_{\tau}(\alpha_{i}^{k})-\sum_{j}\sum_{k=1}^{n_{j}-1}\mu_{\tau}(\beta_{j}^{k}).
\end{equation}
\end{dfn}

\begin{dfn}[]
Let $u=u_{0}\cup{u_{1}}\in \mathcal{M}^{J}(\alpha,\beta)$. Suppose that $u_{1}$ is somewhere injective. Let $n_{i}^{+}$
be the number of positive ends of $u_{1}$ which are asymptotic to $\alpha_{i}$, plus 1 if $u_{0}$ includes the trivial cylinder $\mathbb{R}\times \alpha_{i}$ with some multiplicity. Likewise, let $n_{j}^{-}$ be the number of negative ends of $u_{1}$ which are asymptotic to $\beta_{j}$, plus 1 if $u_{0}$ includes the trivial cylinder $\mathbb{R}\times \beta_{j}$ with some multiplicity. 
\end{dfn}
Write $J_{0}(u)=J_{0}(\alpha,\beta,[u])$.
\begin{prp}[{\cite[Lemma 3.5]{HT3}} {\cite[Proposition 5.8]{H3}}]
Let $\alpha=\{(\alpha_{i},m_{i})\}$ and $\beta=\{(\beta_{j},n_{j})\}$ be admissible orbit sets, and let $u=u_{0}\cup{u_{1}}\in \mathcal{M}^{J}(\alpha,\beta)$. Then
\begin{equation}
    -\chi(u_{1})+\sum_{i}(n_{i}^{+}-1)+\sum_{j}(n_{j}^{-}-1)\leq J_{0}(u)
\end{equation}
If $u$ is counted by the ECH differential or the $U$-map, then the above equality holds. 
\end{prp}

\subsubsection{ECH spectrum}
The notion of ECH spectrum is introduced in \cite{H2}. At first, we consider the filtered ECH.

The action of an orbit set $\alpha=\{(\alpha_{i},m_{i})\}$ is defined by 
\begin{equation*}
A(\alpha)=\sum m_{i}A(\alpha_{i})=\sum m_{i}\int_{\alpha_{i}}\lambda. 
\end{equation*}
For any $L>0$,  $\mathrm{ECC}^{L}(Y,\lambda,\Gamma)$ denotes the subspace of  $\mathrm{ECC}(Y,\lambda,\Gamma)$ which is generated by ECH generators whose actions are less than $L$. Then $(\mathrm{ECC}^{L}(Y,\lambda,\Gamma),\partial_{J})$ becomes a subcomplex and the homology group $\mathrm{ECH}^{L}(Y,\lambda,\Gamma)$ is defined. Here, we use the fact that if two ECH generators $\alpha=\{(\alpha_{i},m_{i})\}$ and $\beta=\{(\beta_{i},n_{i})\}$ have $A(\alpha)\leq A(\beta)$, then the coefficient of $\beta$ in $\partial \alpha$  is $0$ because of the positivity of $J$-holomorphic curves over $d\lambda$ and the fact that $A(\alpha)-A(\beta)$ is equivalent to the integral value of $d\lambda$ over $J$-holomorphic punctured curves which is asymptotic to $\alpha$ at $+\infty$, $\beta$ at $-\infty$.

It follows from the construction that there exists a canonical homomorphism $i_{L}:\mathrm{ECH}^{L}(Y,\lambda,\Gamma) \to \mathrm{ECH}(Y,\lambda,\Gamma)$. In addition, for non-degenerate contact forms $\lambda,\lambda'$ with $\mathrm{Ker}\lambda=\mathrm{Ker}\lambda'=\xi$, there is a canonical isomorphism $\mathrm{ECH}(Y,\lambda,\Gamma)\to \mathrm{ECH}(Y,\lambda',\Gamma)$ defined by  the cobordism maps for product cobordisms (see \cite{H2}). 
Therefore we can consider a pair of a group $\mathrm{ECH}(Y,\xi,\Gamma)$ and maps $j_{\lambda}:\mathrm{ECH}(Y,\lambda,\Gamma) \to \mathrm{ECH}(Y,\xi,\Gamma)$ for  any non-degenerate contact form $\lambda$ with $\mathrm{Ker}\lambda=\xi$ such that $\{j_{\lambda}\}_{\lambda}$ is compatible with the canonical map $\mathrm{ECH}(Y,\lambda,\Gamma)\to \mathrm{ECH}(Y,\lambda',\Gamma)$.
\begin{dfn}[{\cite[Definition 4.1, cf. Definition 3.4]{H2}}]
Let $Y$ be a closed oriented three manifold with a non-degenerate contact form $\lambda$ with $\mathrm{Ker}\lambda=\xi$ and $\Gamma \in H_{1}(Y,\mathbb{Z})$. If $0\neq \sigma \in \mathrm{ECH}(Y,\xi,\Gamma)$, define
\begin{equation*}\label{spect}
    c_{\sigma}^{\mathrm{ECH}}(Y,\lambda)=\inf\{L>0 |\, \sigma \in \mathrm{Im}(j_{\lambda}\circ i_{L}:\mathrm{ECH}^{L}(Y,\lambda,\Gamma) \to \mathrm{ECH}(Y,\xi,\Gamma))\, \}
\end{equation*}
If $\lambda$ is degenerate, define
\begin{equation*}
    c_{\sigma}^{\mathrm{ECH}}(Y,\lambda)=\sup\{c_{\sigma}(Y,f_{-}\lambda)\}=\inf\{c_{\sigma}(Y,f_{+}\lambda)\}
\end{equation*}
where the supremum is over functions $f_{-}:Y\to (0,1]$ such that $f\lambda$ is non-degenerate and the infimum is over smooth functions $f_{+}:Y\to [1,\infty)$ such that $f_{-}\lambda$ is non-degenerate. Note that $c_{\sigma}^{\mathrm{ECH}}(Y,\lambda)<\infty$ and this definition makes sense. See {\cite[Definition 4.1, Definition 3.4, \S 2.3]{H2}} for more details.
\end{dfn}

\begin{dfn}\cite[Subsection 2.2]{H2}
Let $Y$ be a closed oriented three manifold with a non-degenerate contact form $\lambda$ with $\xi$. Then there is a canonical element called the ECH contact invariant,
\begin{equation}
    c(\xi):=\langle \emptyset \rangle \in \mathrm{ECH}(Y,\lambda,0).
\end{equation}
where $\langle \emptyset \rangle$ is the equivalent class containing $\emptyset$.
Note that $\partial_{J} \emptyset =0$ because of the maximal principle. In addition, $c(\xi)$  depends only on the contact structure $\xi$. 
\end{dfn}
\begin{dfn}\label{echspect}\cite[Definition 4.3]{H2}
    If $(Y,\lambda)$ is a closed connected contact three-manifold with $c(\xi)\neq 0$, and if $k$ is a nonnegative integer, define
    \begin{equation}
        c_{k}^{\mathrm{ECH}}(Y,\lambda):=\min\{c_{\sigma}^{\mathrm{ECH}}(Y,\lambda)|\,\sigma\in \mathrm{ECH}(Y,\xi,0),\,\,U^{k}\sigma=c(\xi)\}
    \end{equation}
The sequence $\{c_{k}^{\mathrm{ECH}}(Y,\lambda)\}$ is called the  ECH spectrum of $(Y,\lambda)$.
\end{dfn}
\begin{prp}\label{properties1}
    Let $(Y,\lambda)$ be a closed connected contact three manifold.
    \item[(1).]   \begin{equation*}
         0=c_{0}^{\mathrm{ECH}}(Y,\lambda)<c_{1}^{\mathrm{ECH}}(Y,\lambda) \leq c_{2}^{\mathrm{ECH}}(Y,\lambda)...\leq \infty.
     \end{equation*}
    \item[(2).] For any  $a>0$ and positive integer $k$, 
    \begin{equation*}
         c_{k}^{\mathrm{ECH}}(Y,a\lambda)= ac_{k}^{\mathrm{ECH}}(Y,\lambda).
    \end{equation*}
    \item[(3).] Let $f_{1},f_{2}:Y\to (0,\infty)$ be smooth functions with $f_{1}(x)\leq f_{2}(x)$ for every $x\in Y$. Then \begin{equation*}
        c_{k}^{\mathrm{ECH}}(Y,f_{1}\lambda)\leq c_{k}^{\mathrm{ECH}}(Y,f_{2}\lambda).
    \end{equation*}
    \item[(4).] Suppose $c_{k}^{\mathrm{ECH}}(Y,f\lambda)<\infty$. The map 
    \begin{equation*}
        C^{\infty}(Y,\mathbb{R}_{>0})\ni f \mapsto c_{k}^{\mathrm{ECH}}(Y,f\lambda)\in \mathbb{R}
    \end{equation*}
     is continuous in $C^{0}$-topology on $C^{\infty}(Y,\mathbb{R}_{>0})$.
\end{prp}
\begin{proof}[\bf Proof of Proposition \ref{properties1}]
They follow from the properties of ECH. See \cite{H2}.
\end{proof}
\subsubsection{ECH on Lens spaces}
Now, we focus on lens spaces.

Since $H_{2}(L(p,q))=0$, we write ECH index of $\alpha$, $\beta$ as $I(\alpha,\beta)$ instead of $I(\alpha,\beta,Z)$ where $\{Z\}= H_{2}(L(p,q);\alpha,\beta)$.

Let $(L(p,q),\lambda)$ be a non-degenerate contact lens space and consider ECH of  $0\in H_{1}(L(p,q))$. Since $[\emptyset]=0$, there is  an absolute $\mathbb{Z}$-grading on $\mathrm{ECH}(L(p,q),\lambda,0)=\bigoplus_{k\in\mathbb{Z}}\mathrm{ECH}_{k}(L(p,q),\lambda,0)$ defined by ECH index relative to $\emptyset$ where $\mathrm{ECH}_{k}(L(p,q),\lambda,0)$ is as follows.
Define
\begin{equation*}
    \mathrm{ECC}_{k}(L(p,q),\lambda,0):= \bigoplus_{\alpha:\mathrm{ECH\,\,generator},\,{[\alpha]=0},\,I(\alpha,\emptyset)=k}\mathbb{F}\cdot \alpha.
\end{equation*}
Since  $\partial_{J} $  maps $\mathrm{ECC}_{*}(L(p,q),\lambda,0)$ to  $\mathrm{ECC}_{*-1}(L(p,q),\lambda,0)$ and $U_{J,z}$ does  $\mathrm{ECC}_{*}(L(p,q),\lambda,0)$ to $\mathrm{ECC}_{*-2}(L(p,q),\lambda,0)$, we have $\mathrm{ECH}_{k}(L(p,q),\lambda,0)$ and 
\begin{equation*}
    U:\mathrm{ECH}_{*}(L(p,q),\lambda,0) \to\mathrm{ECH}_{*-2}(L(p,q),\lambda,0).
\end{equation*}
Based on these understandings, the next follows.
\begin{prp}\label{lenssp}
 Let $(L(p,q),\lambda)$ be a non-degenerate.
 \item[(1).]   If $k$ is even and non-negative,
    \begin{equation*}
        \mathrm{ECH}_{k}(L(p,q),\lambda,0)\cong \mathbb{F}.
    \end{equation*}
    If $k$ is odd or negative, $\mathrm{ECH}_{k}(L(p,q),\lambda,0)$ is zero. 
    Moreover, for $n\geq1$ the $U$-map
    \begin{equation*}
        U:\mathrm{ECH}_{2n}(L(p,q),\lambda,0)\to\mathrm{ECH}_{2(n-1)}(L(p,q),\lambda,0)
    \end{equation*}
    is isomorphism.
    \item[(2).] If $\mathrm{Ker}\lambda=\xi_{\mathrm{std}}$, $0\neq c(\xi_{\mathrm{std}})\in \mathrm{ECH}_{0}(L(p,q),\lambda,0)$. Therefore, we can define the ECH spectrum (Definition \ref{echspect}).
\end{prp}
\begin{proof}[\bf Proof of Proposition \ref{lenssp}]
    (1) follows from the isomorphism between ECH and monopole floer homology. See \cite{KM,T1}. (2) follows from the cobordism map between $(L(p,q),\lambda)$ and a quotient of an irrational ellipsoid. This is an analogue of $(S^{3},\xi_{\mathrm{std}})$ in \cite[Proof of Proposition 4.5]{H2}.
\end{proof}

\section{Behaviors of Conley-Zehnder index and $J$-holomorphic curves}

In this section, we observe  behaviors of of Conley-Zehnder index and $J$-holomorphic curves under non-degeneracy. If $(L(p,p-1),\lambda)$ is dynamically convex with $\mathrm{Ker}\lambda=\xi_{\mathrm{std}}$, we can take an universal symplectic trivialization $\tau_{0}:\mathrm{Ker}\lambda=\xi_{\mathrm{std}}\to L(p.p-1)\times \mathbb{C}$. We let $\mu(\gamma)$ denote the Conley-Zehnder index of $\gamma$ defined from $\tau_{0}$. Note that if $\gamma$ is contractible, $\mu_{\mathrm{disk}}(\gamma)=\mu(\gamma)$. 

\begin{lem}\label{conleyindex}
Assume that $(L(p,p-1),\lambda)$ is non-degenerate dynamically convex contact manifold with $\mathrm{Ker}\lambda=\xi_{\mathrm{std}}$.
\item[(1).]  For any orbit $\gamma$, $\mu(\gamma)\geq 1$.
\item[(2).] Suppose that $p=2$. If $\mu(\gamma)=1$, $\gamma$ is a simple non-contractible elliptic orbit and if $\mu(\gamma)=2$, $\gamma$ is a simple non-contractible positive hyperbolic orbit.
\item[(3).]Suppose that $p>2$. If $\mu(\gamma)=1$, $\gamma$ is a simple non-contractible negative hyperbolic orbit otherwise an elliptic orbit which may not be simple.  If $\mu(\gamma)=2$, $\gamma$ is a  non-contractible positive hyperbolic orbit.
\end{lem}

\begin{proof}[\bf Proof of Lemma \ref{conleyindex}]
Suppose $\mu(\gamma)< 3$. Since $\lambda$ is dynamically convex, $\gamma$ is non-contractible. Note that $\gamma^{p}$ is contractible. Suppose that $\gamma$ is hyperbolic. Then $\mu(\gamma^{p})=p\mu(\gamma)\geq 3$. If $p=2$, this implies that  $\mu(\gamma)=2$ and hence $\gamma$ is simple positive hyperbolic.
If $p>2$, $\mu(\gamma^{p})=p\mu(\gamma)\geq 3$ implies that $\mu(\gamma)=1$ or $2$.
 If $\mu(\gamma)=2$, $\gamma$ is a  positive hyperbolic orbit (may not be simple). If $\mu(\gamma)=1$, $\gamma$  is a simple negative hyperbolic orbit. 

Next suppose that $\gamma$ with $\mu(\gamma)< 3$ is elliptic. Then there is $\theta\in \mathbb{R}\backslash \mathbb{Q}$ such that $\mu(\gamma^{n})=2\lfloor n\theta \rfloor +1$ for every $n\in \mathbb{Z}_{>0}$. Since $\mu(\gamma^{p})=2\lfloor p\theta \rfloor +1\geq 3$, we have $\lfloor p\theta \rfloor\geq 1$ and hence $\lfloor \theta \rfloor\geq 0$. In particular, since $\mu(\gamma)< 3$, we have $\lfloor \theta \rfloor= 0$. Moreover  in the case of $p=2$,  $\lfloor 2\theta \rfloor=1$. This implies that if $p=2$, $\gamma$ is simple and $\mu(\gamma)=1$. We complete the proof.
\end{proof}

\begin{lem}\label{emptyindex1}
Assume that $(L(p,p-1),\lambda)$ is non-degenerate dynamically convex contact manifold with $\mathrm{Ker}\lambda=\xi_{\mathrm{std}}$. Let $\alpha$ be an orbit set with $[\alpha]=0\in H_{1}(Y)$.
     If $I(\alpha,\emptyset)=1$, then $\mathcal{M}^{J}(\alpha,\emptyset)=\emptyset$.
\end{lem}
\begin{proof}[\bf Proof of Lemma \ref{emptyindex1}]

Suppose that $\mathcal{M}^{J}(\alpha,\emptyset)\neq \emptyset$. Choose $u \in \mathcal{M}^{J}(\alpha,\emptyset)$ and let $\{\gamma_{i}\}_{i=1,...k}$ denote the set consisting of all orbits to which the positive ends of $u$ are asymptotic. Note that $\gamma_{i}$ may not be simple. Since $u$ is embedded and of Fredholm index 1, we have
    \begin{equation}
        \mathrm{ind}(u)=1=-\chi(u)+\sum_{1\leq i \leq k}\mu(\gamma_{i})=2g-2+\sum_{1\leq i \leq k}(\mu(\gamma_{i})+1).
    \end{equation}
    Since $\mu(\gamma_{i})+1\geq 2$, we have $k=1$. Therefore we  have 
\begin{equation}
    1=2g+\mu(\gamma_{1})-1.
\end{equation}
This implies that $\mu(\gamma_{1})=2$ and $g=0$. It follows from Lemma \ref{conleyindex} that $\gamma_{1}$ is non-contractible and hence $[\alpha]\neq 0$. This contradicts the assumption. We complete the proof.
\end{proof}

\begin{lem}\label{Umap}
Assume that $(L(p,p-1),\lambda)$ is non-degenerate dynamically convex contact manifold with $\mathrm{Ker}\lambda=\xi_{\mathrm{std}}$.
    Let $\alpha$ be an ECH generator with $I(\alpha,\emptyset)=2$ and $[\alpha]=0$. If $\langle U_{J,z}\alpha,\emptyset \rangle \neq 0
$, then there exists an simple orbit  $\gamma$ with $\mu(\gamma)=1$ and $\mu(\gamma^{p})=3$ for which $\alpha=(\gamma,p)$, otherwise there exist  simple orbits $\gamma_{1},\gamma_{2}$ with $\mu(\gamma_{1})=\mu(\gamma_{2})=1$ for which $\alpha=(\gamma_{1},1)\cup (\gamma_{2},1)$. Moreover any element  $u \in \mathcal{M}^{J}(\alpha,\emptyset)$ is of genus $0$ and $\pi(u)$ is a global surface of section for $X_{\lambda}$.
\end{lem}
\begin{proof}[\bf Proof of Lemma \ref{Umap}]
Write $\alpha=\{(\gamma_{i},m_{i})\}_{i=1,...k}$. Let $u \in \mathcal{M}^{J}(\alpha,\emptyset)$ be  a $J$-holomorphic curve counted by $U_{z,J}$. Then  we have
\begin{equation}
    I(u)-J_{0}(u)=\sum_{i=1,...k}\mu(\gamma_{i}^{m_{i}}).
\end{equation}
Therefore
\begin{equation}
    2=-\chi(u)+\sum_{i=1,...k}(n_{i}-1)+\sum_{i=1,...k}\mu(\gamma_{i}^{m_{i}})=2g-2+\sum_{i=1,...k}(\mu(\gamma_{i}^{m_{i}})+2n_{i}-1)
\end{equation}
where $n_{i}$ is the number of positive ends of $u$ asymptotic to $\gamma_{i}$ with some multiplicities and $g$ is the genus of $u$. Since $\mu(\gamma_{i}^{m_{i}})+2n_{i}-1\geq 2$ for $i=1,...k$, we have $k=1$ or $k=2$.

Suppose that $k=1$. Then we have $5=2g+\mu(\gamma_{1}^{m_{1}})+2n_{1}$. Note that $g\geq0$, $m_{1}, n_{1}\geq1$ and $\gamma_{1}$ is simple. 
If $n_{1}\geq 2$, we have $g=0$ and $1=\mu(\gamma_{1}^{m_{1}})$. Since $\lambda$ is dynamically convex, $\gamma_{1}^{m_{1}}$ is a non-contractible elliptic orbit and hence $[\alpha]\neq 0$. This contradicts the assumption. Therefore  $n_{1}=1$ and hence $3=2g+\mu(\gamma_{1}^{m_{1}})$. In the same way, we have $g=0$ and so $3=\mu(\gamma^{m})$.
 In order to complete the proof of Lemma \ref{Umap} in the case of $k=1$, we have to show $m=p$, but since additional discussion is needed, we would leave it for later.

Suppose that $k=2$. Then $6=2g+\mu(\gamma_{1}^{m_{1}})+\mu(\gamma_{2}^{m_{2}})+2n_{1}+2n_{2}$.  Since  $g\geq0$ and $m_{i}, n_{i}\geq1$, we have $n_{1}=n_{2}=1$, $g=0$ and $\mu(\gamma_{1}^{m_{1}})=\mu(\gamma_{2}^{m_{2}})=1$. We have to show $m_{1}=m_{2}=1$.
From Lemma \ref{conleyindex}, if $p=2$, $\gamma_{1}^{m_{1}}$ and $\gamma_{2}^{m_{2}}$ are simple elliptic orbits and hence we have $m_{1}=m_{2}=1$. Therefore we may assume that $p>2$. 
Without loss of generality, we consider $\gamma_{1}$ and $m_{1}$. Suppose  that $m_{1}\geq 2$. By the definition of ECH generator, $\gamma_{1}$ is elliptic. Therefore there is $\theta \in \mathbb{R}\backslash \mathbb{Q}$ such that $\mu(\gamma_{1}^{m})=2 \lfloor m\theta \rfloor+1$ for every $m\in \mathbb{Z}_{>0}$. 
Since $\mu(\gamma_{1}^{m_{1}})=2 \lfloor m_{1}\theta \rfloor+1=1$, we have  $2 \lfloor i\theta \rfloor+1=1$ for every $1\leq i \leq m_{1}$. Hence $\lfloor m_{1}\theta \rfloor=m_{1}\lfloor \theta \rfloor$. It follows from the partition condition (Proposition \ref{partitioncondition}) that each multiplicity of positive end of $u$ asymptotic to $\gamma_{1}$ is $1$ and hence the number of ends asymptotic to $\gamma_{1}$ is $m_{1}$. This implies that $n_{1}=m_{1} \geq 2$. This contradicts $n_{1}=1$. Therefore we have $m_{1}=1$. In the same way, we have $m_{2}=1$.

Next, we show that for any element  $u \in \mathcal{M}^{J}(\alpha,\emptyset)$,  $\pi(u)$ is a global surface of section for $X_{\lambda}$.
Since $I(u)=2$, $u\in \mathcal{M}^{J}(\alpha,\emptyset)$ is embedded.

According to \cite[Proposition 3.2]{CHP}, in order to prove that $\pi(u)$ is a global surface of section, it is sufficient to show that $\mathcal{M}^{J}(\alpha,\emptyset)/\mathbb{R}$ is compact. Indeed, if $\mathcal{M}^{J}(\alpha,\emptyset)/\mathbb{R}$ is compact, $u \in \mathcal{M}^{J}(\alpha,\emptyset)$ satisfies all of the  assumptions of  \cite[Proposition 3.2]{CHP}. 

Suppose that $\mathcal{M}^{J}(\alpha,\emptyset)/\mathbb{R}$ is not compact. Let $\overline{\mathcal{M}^{J}(\alpha,\emptyset)/\mathbb{R}}$ denote the compactified space  of $\mathcal{M}^{J}(\alpha,\emptyset)/\mathbb{R}$ in the sense of SFT compactness. Choose $\Bar{u}\in \overline{\mathcal{M}^{J}(\alpha,\emptyset)/\mathbb{R}} \backslash \mathcal{M}^{J}(\alpha,\emptyset)/\mathbb{R}$. $\Bar{u}$ consists of some $J$-holomorphic curves in several floors. Let $u_{0}$ be the component of $\Bar{u}$ in the lowest floor. Then there is an orbit set $\beta$ such that $u_{0}\in \mathcal{M}^{J}(\beta,\emptyset)/\mathbb{R}$. By the additivity of ECH index, we have $I(\beta,\emptyset)=1$. This contradicts Lemma \ref{emptyindex1}. So $\mathcal{M}^{J}(\alpha,\emptyset)/\mathbb{R}$ is compact.
    
Summarizing the discussion so far, under the assumption of Lemma \ref{Umap}, we have
\item[(1).]There is a contractible simple orbit $\gamma$ such that $\alpha=(\gamma,m)$ and $3=\mu(\gamma^{m})$. Moreover any $u\in \mathcal{M}^{J}(\alpha,\emptyset)$ is of genus $0$ and $\pi(u)$ is a global surface of section whose boundary is $\gamma^{m}$.

\item[(2).]There is non-contractible simple orbits $\gamma_{1},\gamma_{2}$ such that $\alpha=(\gamma_{1},1)\cup(\gamma_{2},1)$ and $1=\mu(\gamma_{1})=\mu(\gamma_{2})$. Moreover any $u\in \mathcal{M}^{J}(\alpha,\emptyset)$ is of genus $0$ and $\pi(u)$ is a global surface of section whose boundary is $\gamma_{1}\cup \gamma_{2}$.

    Now, it is sufficient to show that $m=p$ in (1) to complete the proof.  Take a section $s:\mathcal{M}^{J}(\alpha,\emptyset)/\mathbb{R}\to \mathcal{M}^{J}(\alpha,\emptyset)$. Then $\bigcup_{\tau\in \mathcal{M}^{J}(\alpha,\emptyset)/\mathbb{R}}\overline{\pi(s(\tau))} \to \mathcal{M}^{J}(\alpha,\emptyset)/\mathbb{R}$ is an (rational) open book decomposition of $L(p,p-1)$ (see \cite[Proposition 3.2, Proposition 3.3]{CHP}). Note that $\mathcal{M}^{J}(\alpha,\emptyset)/\mathbb{R}\cong S^{1}$. It is easy to compute that the fundamental group of $\bigcup_{\tau\in \mathcal{M}^{J}(\alpha,\emptyset)/\mathbb{R}}\overline{\pi(s(\tau))}$ is isomorphic to $\mathbb{Z}/m\mathbb{Z}$, which means $m=p$. This completes the proof.
\end{proof}
 Let $Z:L(p,p-1)\to \xi_{\mathrm{std}}$ denote a non-vanishing global section satisfying $\tau_{0}(Z(x))=(x,1)\in L(p,p-1)\times \mathbb{C}$ where $\tau_{0}$ is the fixed global trivialization.
 
 For a simple orbit $\gamma:\mathbb{R}/T_{\gamma}\mathbb{Z}\to L(p,p-1)$ (which is parameterized) and sufficiently small $\epsilon_{\gamma}>0$ there is an open neighborhood $\gamma(\mathbb{R}/T_{\gamma}\mathbb{Z})\subset U_{\gamma} \subset L(p,p-1)$, an open neighborhood $0 \in V_{\gamma}\subset \mathbb{R}^{2}$ and a diffeomorphism $\phi_{\gamma}:S^{1} \times V_{\gamma}\to U_{\gamma}$ satisfying  $\phi_{\gamma}(t,0)=\gamma_{i}(T_{\gamma}t)$ and $\phi_{\gamma}^{*}\lambda=f_{\gamma}\lambda_{0}$
where $f_{\gamma}$ is a positive map $f_{\gamma}:S^{1} \times V_{\gamma} \to \mathbb{R}$ with $f_{\gamma}(t,0,0)=T_{\gamma}$, $df_{\gamma}(t,0,0)=0$ and $\lambda_{0}$ denotes the standard contact structure on $\mathbb{R}^{3}$. In addition, we can take $\phi_{\gamma}$ satisfying $(\phi_{\gamma})_{*}(\partial_{x})=Z$ on $\gamma$. See \cite{HWZ1}. 

It is convenient to see $J$-holomorphic curves in this coordinate.

\begin{prp}\label{asymptics}\cite[cf.Theorem 1.3]{HWZ1}
    Let $(a,u):([0,\infty)\times S^{1},j_{0})\to (\mathbb{R}\times L(p,p-1),J)$ is a $J$-holomorphic curve such that $u(s,t)\to\gamma(pT_{\gamma}t)$ for $p>0$ where $j_{0}(\partial_{s})=\partial_{t}$. Take a coordinate  $\phi_{\gamma}:S^{1} \times V_{\gamma}\to U_{\gamma}$ defined above and write $\phi^{-1}_{\gamma}(u(s,t))=(\theta(s,t),z(s,t))$
    If $\gamma^{p}$ is non-degenerate, there is constant $c$ and $d>0$ such that
    \begin{equation*}
        |\partial^{\beta}(a(s,t)-T_{\gamma}s-c)|\leq Me^{-ds},\,\,|\partial^{\beta}(\theta(s,t)-pt)|\leq Me^{-ds}.
    \end{equation*}
    for all multi-indices $\beta$ with constants $M=M_{\beta}$.
    
    Moreover, there is negative eigenvalue $\lambda\in \sigma(L_{S_{\phi_{\gamma,\tau}\circ \rho_{k}}})$, eigenfunction $e_{\lambda}$ of $L_{S_{\phi_{\gamma,\tau}\circ \rho_{k}}}$ with eigenvalue $\lambda$, and  sufficiently small $\epsilon>0$  such that
    \begin{equation*}
        |\partial^{\beta}(z(s,t)-e^{\lambda s}e_{\lambda}(t))|\leq Me^{\lambda-\epsilon}.
    \end{equation*}
    where $L_{S_{\phi_{\gamma,\tau}\circ \rho_{k}}}$ is the operator defined in subsection 2.1.
\end{prp}

\begin{lem}\label{selflinkingproof1}
    Suppose that $\alpha$ in Lemma \ref{Umap}  can be written as $\alpha=(\gamma,p)$ for a simple orbit $\gamma$. Then $\gamma \in \mathcal{S}_{p}$. 
\end{lem}
\begin{proof}[\bf Proof of Lemma \ref{selflinkingproof1}]
Since  for $ u \in \mathcal{M}^{J}(\alpha,\emptyset)$ $\pi(u)\subset Y$ is embedded, $\gamma $ is $p$-unknotted and we can determine the self-linking number by $\pi(u)$. Let $\phi_{\gamma}:S^{1}\times \mathbb{D}_{\epsilon_{\gamma}}\to L(p,p-1)$ be a local coordinate defined above. According to \cite[Proof of Lemma 3.3]{CHP}, the winding number of the eigenfunction dominating $u$ as in Proposition \ref{asymptics} is $1$. This means that the algebraic intersection number of  $\mathrm{exp}_{\gamma(pT_{\gamma}t)}(\epsilon Z(\gamma(pT_{\gamma}t)))$ with $\pi(u)$ is $-p$.
Therefore by the definition of (rational) self linking number, we have $\mathrm{sl}(\gamma)=-\frac{1}{p}$. This completes the proof.
\end{proof}

\begin{lem}\label{muzukasiiyo2}
     Suppose that $\alpha$ in Lemma \ref{Umap}  can be written as  $\alpha=(\gamma_{1},1)\cup (\gamma_{2},1)$ for simple orbits $\gamma_{1},\gamma_{2}$. Then $\gamma_{1},\gamma_{2} \in \mathcal{S}_{p}$
\end{lem}

In order to prove Lemma \ref{muzukasiiyo2}, we fix the parameters $\gamma_{1}:\mathbb{R}/T_{1}\mathbb{Z}\to Y$, $\gamma_{2}:\mathbb{R}/T_{2}\mathbb{Z}\to Y$. Here we write the periods of $\gamma_{1}$, $\gamma_{2}$ as  $T_{1}$, $T_{2}$ for simplicity. 

The next theorem is used to construct Seifert surface to compute their self-linking number and will be roved in Section 5.

\begin{them}\label{familyofhol}
Assume that $(L(p,p-1),\lambda)$ is non-degenerate dynamically convex contact manifold with $\mathrm{Ker}\lambda=\xi_{\mathrm{std}}$.
    Let $\alpha$ be an ECH generator with $I(\alpha,\emptyset)=2$ and $[\alpha]=0$. Suppose that $\langle U_{J,z}\alpha,\emptyset \rangle\neq 0$ and $\alpha$   can be written as  $\alpha=(\gamma_{1},1)\cup (\gamma_{2},1)$ for simple orbits $\gamma_{1},\gamma_{2}$.
Then there is a smooth map $\Tilde{u}=(a,u):S^{1}\times \mathbb{R}\times S^{1}\to \mathbb{R}\times Y$ satisfying the following properties. 

\item[1.] $u:S^{1}\times \mathbb{R}\times S^{1}\to Y$ is  embedding.
\item[2.] For any $\tau\in S^{1}$, $\Tilde{u}(\tau,\cdot,\cdot)=\Tilde{u}_{\tau}=(a_{\tau},u_{\tau})\in \mathcal{P}^{J}$ with $u(s,t)\to\gamma_{1}(T_{1}(t+p\tau))$ as $s\to +\infty$ and $u(s,t)\to\gamma_{1}(-T_{2}t)$ as $s\to -\infty$
\item[3.]   $S^{1}\ni \tau \mapsto [\Tilde{u}_{\tau}]\in \mathcal{M}^{J}(\alpha,\emptyset)/\mathbb{R}$ is a diffeomoprhism map where $ [\Tilde{u}_{\tau}]$ is the equivalent class containing $\Tilde{u}_{\tau}$.
\end{them}

\begin{proof}[\bf Proof of Lemma \ref{muzukasiiyo2}]
Without loss of generality, we consider $\gamma_{1}$.  Let $\Tilde{u}=(a,u):S^{1}\times \mathbb{R}\times S^{1}\to \mathbb{R}\times Y$  be as in Theorem \ref{familyofhol}. We define a continuous map $v:\mathbb{C} \to L(p,p-1)$ by $v(re^{2\pi\tau})=u_{\tau}(\mathrm{log}r,0)$ and $v(0)=\gamma_{2}(0)$. Then by the construction, we have $v(re^{2\pi\tau})\to \gamma_{1}(pT_{1}\tau)$ as $r\to +\infty$ and moreover this map is immersion and injective other than $0$ and  $\gamma_{1}$. In order to obtain Seifert surface to compute the self-linking number, we change $v$ near $\gamma_{2}(0)$ and $\gamma_{1}$

Take the local coordinates $\phi_{\gamma_{1}}:S^{1}\times \mathbb{D}_{\epsilon_{\gamma_{1}}}\to L(p,p-1)$ as before. Fix $\tau \in S^{1}$. According to \cite[Proof of Lemma 3.3]{CHP}, for any sufficiently small $\epsilon>0$, the winding number of the eigenfunction dominating $u_{\tau}$ as in Proposition \ref{asymptics} is $0$. Therefore, by moving $\tau \in S^{1}$, we can see that  for any sufficiently small $\epsilon>0$, $\phi_{\gamma_{1}}^{-1}(\phi_{\gamma_{1}}(S^{1}\times \partial \mathbb{D}_{\epsilon})\cap{\pi(v(\mathbb{C}))})$ is $(p,1)$-torus knot or $(p,-1)$-torus knot. 

Fix a sufficiently small $\epsilon>0$.  We consider a surface constructed by connecting $\phi_{\gamma_{1}}(S^{1}\times \partial \mathbb{D}_{\epsilon})\cap{\pi(v(\mathbb{C}))}$ with $\gamma_{1}$ by an isotopy. And we connect  it smoothly with the outer part of $v(\mathbb{C})$ together and change small neighborhood of $v(0)$ to be embedded. We write the surface as $F\subset L(p,p-1)$. By the construction, this is a (rational) Seifert surface of disk type binding $\gamma_{1}$ with multiplicity $p$. Moreover by the construction $\mathrm{sl}(\gamma_{1})=-\frac{1}{p}$ or $\frac{1}{p}$ which depends on whether $\phi_{\gamma_{1}}^{-1}(\phi_{\gamma_{1}}(S^{1}\times \partial \mathbb{D}_{\epsilon})\cap{\pi(v(\mathbb{C}))})$ is $(p,1)$-torus knot or $(p,-1)$-torus knot.  On the other hand, since $\xi_{\mathrm{std}}$ is universally tight and $\gamma_{1}$ is transversal knot, we can apply 
 (rational) Bennequin inequality to this (see \cite{BE}). Therefore we have $\mathrm{sl}(\gamma_{1})\leq -\frac{1}{p}\chi(F)=-\frac{1}{p}$. This implies $\mathrm{sl}(\gamma_{1})=-\frac{1}{p}$. We complete the proof.
\end{proof}

\section{Proof of Theorem \ref{maintheorem} in non-degenerate cases}
The purpose of this section is to prove Theorem \ref{maintheorem} under non-degeneracy condition.
It is easy to check that Theorem \ref{maintheorem}  under non-degenerate condition follows from the next propositions. We prove them respectively.
\begin{prp}\label{prp1}
    Assume that $(L(p,p-1),\lambda)$ is non-degenerate and dynamically convex with $\mathrm{Ker}\lambda=\xi_{\mathrm{std}}$. Then there exists  $\gamma \in \mathcal{S}_{p}$ satisfying $\mu(\gamma)=1$ and
\begin{equation}
    \int_{\gamma}\lambda\leq \frac{1}{2}\,c_{1}^{\mathrm{ECH}}(L(p,p-1),\lambda).
\end{equation}
\end{prp}
\begin{prp}\label{prp2}
    Assume that $(L(2,1),\lambda)$ is non-degenerate and dynamically convex. Then
    \begin{equation}
    \frac{1}{2}\,c_{1}^{\mathrm{ECH}}(L(2,1),\lambda) \leq \inf_{\gamma\in \mathcal{S}_{2},\mu(\gamma)=1}\int_{\gamma}\lambda
\end{equation}
\end{prp}

\subsection{Proof of Proposition \ref{prp1}}
For a sum of ECH generators $\alpha_{1}+...+\alpha_{k}$ with $\partial_{J}(\alpha_{1}+...+\alpha_{k})=0$, we write $\langle \alpha_{1}+...+\alpha_{k} \rangle$ as the equivalence class in $ECH(Y,\lambda)$.

\begin{lem}\label{existenceofhol}
Let $\alpha_{1},...,\alpha_{k}$ be ECH generators with $[\alpha_{i}]=0$ and $I(\alpha_{i},\emptyset)=2$ for $i=1,...k$. Suppose that $\partial_{J}(\alpha_{1}+...+\alpha_{k})=0$ and $0 \neq \langle \alpha_{1}+...+\alpha_{k} \rangle 
\in  ECH_{2}(Y,\lambda,0)$. Then there exists $i$ such that $\langle U_{J,z}\alpha_{i},\emptyset \rangle\neq  0$.
\end{lem}
\begin{proof}[\bf Proof of Lemma \ref{existenceofhol}]
Since $\mathrm{Ker}\lambda=\xi_{\mathrm{std}}$, $\emptyset+\mathrm{Im}(\partial_{J}|_{ECC_{1}(Y,\lambda,0)})$ is not zero in $ECH_{0}(Y,\lambda,0)$. It follows from Lemma \ref{emptyindex1} that if non zero element in $ECH_{0}(Y,\lambda,0)$ is represented as $\beta_{1}+...+\beta_{j}+\mathrm{Im}(\partial_{J}|_{ECC_{1}(Y,\lambda,0)})$ for some ECH generators $\beta_{1},...,\beta_{j}$, there is $k$ such that $\beta_{k}=\emptyset$ and $\sum_{i\neq k}\beta_{i}$ is in $\mathrm{Im}(\partial_{J}|_{ECC_{1}(Y,\lambda,0)})$. Indeed, if $\beta_{k}\neq \emptyset$ for every $1\leq k \leq j$, $\emptyset+\mathrm{Im}(\partial_{J}|_{ECC_{1}(Y,\lambda,0)})$ and $\beta_{1}+...+\beta_{j}+\mathrm{Im}(\partial_{J}|_{ECC_{1}(Y,\lambda,0)})$ are linearly independent in $ECH_{0}(Y,\lambda,0)$. Otherwise since $ECH_{2}(Y,\lambda,0)\cong \mathbb{F}$, $\emptyset+\beta_{1}+...+\beta_{j}+\mathrm{Im}(\partial_{J}|_{ECC_{1}(Y,\lambda,0)})$ is zero in $ECH_{0}(Y,\lambda,0)$ and so $\emptyset+\beta_{1}+...+\beta_{j}\in \mathrm{Im}(\partial_{J}|_{ECC_{1}(Y,\lambda,0)})$. This implies that there is an ECH generator $\sigma$ with $I(\sigma,\emptyset)=1$ such that $\mathcal{M}^{J}(\sigma,\emptyset)\neq \emptyset$, but this contradicts Lemma \ref{emptyindex1}.

Since $U:ECH_{2}(Y,\lambda,0)\to ECH_{0}(Y,\lambda,0)$ is isomorphism, we can see that $U_{J,z}(\alpha_{1}+...+\alpha_{k})+\mathrm{Im}(\partial_{J}|_{ECC_{1}(Y,\lambda,0)})$ is not zero in $ECH_{0}(Y,\lambda,0)$. Summarizing these arguments, if $\langle U_{J,z}\alpha_{i},\emptyset \rangle= 0$ for any $i$, we have $U_{J,z}(\alpha_{1}+...+\alpha_{k})\in \mathrm{Im}(\partial_{J}|_{ECC_{1}(Y,\lambda,0)})$. Therefore there is $i$ such that $\langle U_{J,z}\alpha_{i},\emptyset \rangle\neq  0$. We complete the proof.
\end{proof}

\begin{proof}
    Now we prove Proposition \ref{prp1}.  From the definition of $c_{1}^{\mathrm{ECH}}$, there are ECH generators $\alpha_{1},...,\alpha_{k}$ with $[\alpha_{i}]=0$, $I(\alpha_{i},\emptyset)=2$ and $A(\alpha_{i})\leq c_{1}^{\mathrm{ECH}}(L(p.p-1),\lambda)$ for $i=1,...k$ such that $\partial_{J}(\alpha_{1}+...+\alpha_{k})=0$ and $0 \neq \langle \alpha_{1}+...+\alpha_{k} \rangle 
\in  ECH_{2}(Y,\lambda,0)$. By Lemma \ref{existenceofhol}, we can find $i$ such that $\langle U_{J,z}\alpha_{i},\emptyset \rangle\neq  0$.  By combining with Lemma \ref{Umap} and Lemma \ref{muzukasiiyo2}, there are either $\gamma\in \mathcal{S}_{p}$ with $\mu(\gamma)=1$ for which $\alpha_{i}=(\gamma,p)$ or $\gamma_{1},\gamma_{2}\in \mathcal{S}_{p}$ with $\mu(\gamma_{1})=\mu(\gamma_{2})=1$ for which $\alpha_{i}=(\gamma_{1},1)\cup(\gamma_{2},1)$. In any cases, this implies that there is $\gamma\in \mathcal{S}_{p}$ with $\mu(\gamma)=1$ such that $ c_{1}^{\mathrm{ECH}}(L(p,p-1),\lambda)\geq 2\int_{\gamma}\lambda$. So we have $ \frac{1}{2}\,c_{1}^{\mathrm{ECH}}(L(p,p-1),\lambda)\geq \int_{\gamma}\lambda$. This completes the proof.
\end{proof}

\subsection{Proof of Proposition \ref{prp2}}
In this subsection, we focus on the case of $p=2$.
\begin{lem}\label{index2}
    For any $\gamma \in \mathcal{S}_{2}$ with $\mu(\gamma)=1$, $I((\gamma,2),\emptyset)=2$.
\end{lem}
\begin{proof}[\bf Proof of Lemma \ref{index2}]
    Consider ECH index $I((\gamma,2),\emptyset)$ by the global trivialization $\tau_{0}:\mathrm{Ker}\lambda=\xi_{\mathrm{std}}\to L(2.1)\times \mathbb{C}$. It easily follows from the global trivialization, the term of the relative first chern number in $I((\gamma,2),\emptyset)$ is $0$.
    
    Next consider the term $Q_{\tau_{0}}$. Refer to \cite[Lemma 8.5]{H1} for the definition of $Q_{\tau_{0}}$.   Since we can take two global surface of sections binding $\gamma$ in $L(2,1)$ which have no intersection point in their interior and moreover the rational self linking number is $-\frac{1}{2}$,  we have $Q_{\tau_{0}}=-2$.
    
    Finally, the term of Conley Zehnder index is 4 because $\mu(\gamma)=1$ and $\mu(\gamma^{2})=3$.

    Summarizing, it follows that $I((\gamma,2),\emptyset)=0+(-2)+4=2$.
\end{proof}

\begin{lem}\label{cocycle}
 For $\gamma\in \mathcal{S}_{2}$ with $\mu(\gamma)=1$, let $\alpha_{\gamma}=(\gamma,2)$.
     Then for any ECH generator $\beta$ with $I(\alpha_{\gamma},\beta)=1$, $\mathcal{M}^{J}(\alpha_{\gamma},\beta)=\emptyset$. In particular,  $\partial_{J}\alpha_{\gamma}=0$.
\end{lem}
\begin{proof}[\bf Proof of Lemma \ref{cocycle}]
     Suppose $\mathcal{M}^{J}(\alpha_{\gamma},\beta)\neq \emptyset$. Take $u\in \mathcal{M}^{J}(\alpha_{\gamma},\beta)$ and regard it as embedded $u\subset \mathbb{R}\times L(2,1)$ and consider $\pi(u\cap[-R,R]\times L(2,1)) \subset L(2,1)$ for large $R>>0$ where $\pi$ is the projection to $L(2,1)$.  Take  $v:\mathbb{D}\to L(2,1)$  a global surface of section  binding $\gamma^{2}$. Note that $H_{1}(L(2,1)\backslash \gamma)\cong \mathbb{Z}$. It follows easily from the topological observations that
     \begin{equation}\label{intersectioncounting}
        \#( \pi(u\cap\{R\}\times L(2,1))\cap v(\mathbb{D}))= \#( \pi(u\cap[-R,R]\times L(2,1)) \cap \gamma )+\# ( \pi(u\cap\{-R\}\times L(2,1))\cap v(\mathbb{D})).
     \end{equation}
      where $\#$ counts  intersection points algebraically.
Considering the partition condition Proposition \ref{partitioncondition}, it follows that the positive end asymptotic to $\gamma$ of $u$ is exactly one and the multiplicity is $2$. Therefore $u$ has no trivial cylinder.
    Consider in  a coordinate  $\phi_{\gamma}:S^{1} \times V_{\gamma}\to U_{\gamma}$ defined before Proposition \ref{asymptics}. Since it follows from the monotonicity of  winding numbers of the eigenfunctions with respect to eigenvalues, the winding number of the eigenfunction dominating the positive end asymptotic to $\gamma$ of $u$ as in Proposition \ref{asymptics} is at most $1$. Hence  $\#( \pi(u\cap\{R\}\times L(2,1))\cap v(\mathbb{D}))$ is not positive. 
    
    Next, consider $\# ( \pi(u\cap\{-R\}\times L(2,1))\cap v(\mathbb{D}))$. Assume that $\beta$ contains $\gamma$. By considering $(s,t)\mapsto (-s,t)$ in Proposition \ref{asymptics}, we can apply Proposition \ref{asymptics} to  negative ends asymptotic to $\gamma$ with some multiplicities. In addition, the number of the negative ends of $u$ asymptotic to $\gamma$ is $1$ and the multiplicity is $1$. Otherwise $\int u^{*}d\lambda=A(\alpha_{\gamma})-A(\beta)\leq A(\alpha_{\gamma})-2A(\gamma)\leq 0$. This is contradiction. Now it follows from Proposition \ref{asymptics} that  the winding number of the eigenfunction dominating the negative end asymptotic to $\gamma$ of $u$ as in Proposition \ref{asymptics} is at least $1$ and hence $\pi(u\cap\{-R\}\times L(2,1))$ rotates one while during one circle in the direction of the orbit. On the other hand, since $v(\mathbb{D})$ is a global surface of section, the intersection number of  any periodic orbit with $v(\mathbb{D})$ is positive.  This implies that $\# ( \pi(u\cap\{-R\}\times L(2,1))\cap v(\mathbb{D}))$ is always positive.

    Finally, consider $\#( \pi(u\cap[-R,R]\times L(2,1)) \cap \gamma )$. it follows from positivity of intersection numbers of $J$-holomorphic curves in 4-dimension that $\#( \pi(u\cap[-R,R]\times L(2,1)) \cap \gamma )$ is not negative.

    Summarizing, the left hand side of (\ref{intersectioncounting}) is not positive but the right hand side is positive. This is a contradiction. Therefore we have $\mathcal{M}^{J}(\alpha_{\gamma},\beta)=\emptyset$.
\end{proof}
\begin{lem}\label{Umapnonzero}
For $\gamma\in \mathcal{S}_{2}$ with $\mu(\gamma)=1$, let $\alpha_{\gamma}=(\gamma,2)$. Then  $\langle U_{J,z} \alpha_{\gamma},\emptyset \rangle\neq 0 $  for generic $z\in L(2,1)$.
\end{lem}
\begin{proof}[\bf Proof of Lemma \ref{Umapnonzero}]
Recall that each page of the rational open book decomposition constructed in \cite[Theorem 1.7, Corollary 1.8]{HrS} is the projection of $J$-holomorphic curve from $(\mathbb{C},i)$ to $L(2,1)$. Moreover in this case,  $:\mathcal{M}^{J}(\alpha_{\gamma},\emptyset)/\mathbb{R}$ is compact and  any two distinct elements $u_{1},u_{2}\in \mathcal{M}^{J}(\alpha_{\gamma},\emptyset)$ has no intersection point. 
Hence $:\mathcal{M}^{J}(\alpha_{\gamma},\emptyset)/\mathbb{R}\cong S^{1}$ and
     for a section $s:\mathcal{M}^{J}(\alpha_{\gamma},\emptyset)/\mathbb{R}\to \mathcal{M}^{J}(\alpha_{\gamma},\emptyset)$, $\bigcup_{\tau\in \mathcal{M}^{J}(\alpha_{\gamma},\emptyset)/\mathbb{R}}\overline{\pi(s(\tau))} \to \mathcal{M}^{J}(\alpha_{\gamma},\emptyset)/\mathbb{R}$ is an (rational) open book decomposition of $L(2,1)$. This implies that for $z\in L(2,1)$ not on $\gamma$, there is exactly one $J$-holomorphic curve in $ \mathcal{M}^{J}(\alpha_{\gamma},\emptyset)$ through $(0,z)\in \mathbb{R}\times L(2,1)$. Therefore we have $\langle U_{J,z} \alpha_{\gamma},\emptyset \rangle\neq 0 $.
\end{proof}

Define a set $\mathcal{G}$ consisting of ECH generators as
\begin{equation}
    \mathcal{G}:=\{\alpha=(\gamma_{1},1)\cup (\gamma_{2},1)|\gamma_{1},\gamma_{2}\in \mathcal{S}_{2},I(\alpha,\emptyset)=2,\mathcal{M}^{J}(\alpha,\emptyset)\neq 0\}\cup \{(\gamma,2)|\gamma \in \mathcal{S}_{2},\mu(\gamma)=1\}.
\end{equation}
Note that $\langle U_{J,z} \alpha,\emptyset \rangle\neq 0 $  if and only if  $\alpha \in \mathcal{G}$.
\begin{lem}\label{index1bound}
Suppose that $\beta$ is an ECH generator with $I(\beta,\alpha)=1$.
Then 
\begin{equation}
    \sum_{\alpha\in \mathcal{G}}\langle \partial_{J}\beta,\alpha \rangle=0
\end{equation}
\end{lem}
\begin{proof}[\bf Proof of Lemma \ref{index1bound}]
Write
\begin{equation}
    \partial_{J}\beta=\sum_{\alpha\in \mathcal{G}}\langle \partial_{J}\beta,\alpha \rangle \alpha+\sum_{I(\beta,\sigma)=1,\sigma\notin \mathcal{G}}\langle \partial_{J}\beta,\sigma \rangle \sigma.
\end{equation}
Then we have
\begin{equation}
   \langle U_{J,z}\partial_{J}\beta,\emptyset \rangle=\sum_{\alpha\in \mathcal{G}}\langle \partial_{J}\beta,\alpha \rangle \langle U_{J,z} \alpha,\emptyset \rangle +\sum_{I(\beta,\sigma)=1,\sigma\notin \mathcal{G}}\langle \partial_{J}\beta,\sigma \rangle \langle \sigma,\emptyset \rangle=\sum_{\alpha\in \mathcal{G}}\langle \partial_{J}\beta,\alpha \rangle
\end{equation}
Here we use that for $\alpha\in \mathcal{G}$, $\langle U_{J,z} \alpha,\emptyset \rangle=1$ and for $\sigma$ with $\sigma \notin \mathcal{G}$, $\langle U_{J,z} \sigma,\emptyset \rangle=0$.

Since $U_{J,z}\partial_{J}=\partial_{J}U_{J,z}$, we have $\langle U_{J,z}\partial_{J}\beta,\emptyset \rangle=\langle \partial_{J}U_{J,z}\beta,\emptyset \rangle=0$ (Here we use  Lemma \ref{emptyindex1}). This completes the proof.
\end{proof}

\begin{lem}\label{nonzero}
     For any $\gamma \in \mathcal{S}_{2}$ with $\mu(\gamma)=1$,  $0\neq \langle \alpha_{\gamma} \rangle = \langle(\gamma,2) \rangle \in ECH_{2}(Y,\lambda,0)$.
\end{lem}
\begin{proof}[\bf Proof of Lemma \ref{nonzero}]
Suppose that $0=\langle \alpha_{\gamma} \rangle  \in ECH_{2}(Y,\lambda,0)$.
Then there are ECH generators $\beta_{1},...\beta_{j}$ with $I(\beta_{i},\alpha_{\gamma})=1$ for any $i$ such that $\partial_{J}(\beta_{1}+...+\beta_{j})=\alpha_{\gamma}$. From Lemma \ref{index1bound}, we have 
\begin{equation}
    \sum_{1\leq i \leq j} \sum_{\alpha\in \mathcal{G}}\langle \partial_{J}\beta_{i},\alpha \rangle=\sum_{\alpha\in \mathcal{G}}\langle \alpha_{\gamma},\alpha \rangle=0.
\end{equation}
But since $\alpha_{\gamma}\in \mathcal{G}$,  $\sum_{\alpha\in \mathcal{G}}\langle \alpha_{\gamma},\alpha \rangle =1$. This is a contradiction. We complete the proof.
\end{proof}

\begin{proof}
Now, we prove Proposition \ref{prp2}. From Lemma \ref{nonzero} and  the definition of $c_{1}^{\mathrm{ECH}}(L(2,1),\lambda)$, we have $2\int_{\gamma} \lambda \geq c_{1}^{\mathrm{ECH}}(L(2,1),\lambda)$ for any $\gamma \in \mathcal{S}_{2}$ with $\mu(\gamma)=1$. This implies $\inf_{\gamma\in \mathcal{S}_{2},\mu(\gamma)=1}\int_{\gamma}\lambda\geq \frac{1}{2}c_{1}^{\mathrm{ECH}}(L(2,1),\lambda)$. This completes the proof.
\end{proof}

\section{Extend the results to degenerate cases}
\subsection{Case of $p=2$}
In this subsection, we prove Theorem \ref{maintheorem} (A) under degenerate strictly convex as a limiting case of non-degenerate result. At first, we show;
\begin{prp}\label{degeneratel2}
     Assume that $(L(2,1),\lambda)$ is strictly convex. Then there exists a simple orbit $\gamma\in \mathcal{S}_{2}$ such that $\mu(\gamma)=1$ and $\int_{\gamma}\lambda = \frac{1}{2}\,c_{1}^{\mathrm{ECH}}(L(2,1),\lambda)$. In particular,
    \begin{equation}
   \inf_{\gamma\in \mathcal{S}_{2},\mu(\gamma)=1}\int_{\gamma}\lambda \leq \frac{1}{2}\,c_{1}^{\mathrm{ECH}}(L(2,1),\lambda).
\end{equation}
\end{prp}
\begin{proof}[\bf Proof of Proposition \ref{degeneratel2}]
    Let $L=c^{\mathrm{ECH}}_{1}(L(2,1),\lambda)$. Take a sequence of strictly convex contact forms $\lambda_{n}$ such that $\lambda_{n} \to \lambda$ in $C^{\infty}$-topology and $\lambda_{n}$ is non-degenerate for each $n$. 
    Therefore we have
    \begin{equation}
        \inf_{\gamma\in \mathcal{S}_{2},\mu(\gamma)=1}\int_{\gamma}\lambda_{n} = \frac{1}{2}\,c_{1}^{\mathrm{ECH}}(L(2,1),\lambda_{n})
    \end{equation}
    Note that  $c_{1}^{\mathrm{ECH}}(L(2,1),\lambda_{n})\to L$ as $n\to+\infty$. This means that there is a sequence of $\gamma_{n}\in \mathcal{S}_{2}(L(2,1),f_{n}\lambda)$ with $\mu(\gamma_{n})=1$ such that $\int_{\gamma_{n}}\lambda_{n}\to \frac{1}{2}L$. By Arzelà–Ascoli theorem, we can find a subsequence which converges to a periodic orbit $\gamma$ of $\lambda$  in $C^{\infty}$-topology. 
    \begin{cla}\label{limittingorbit}
       $\gamma$ is simple. In particular, $\gamma\in \mathcal{S}_{2}(L(2,1),\lambda)$ and $\mu(\gamma)=1$.
    \end{cla}
    \begin{proof}[\bf Proof of Claim \ref{limittingorbit}]
        By the argument so far, there is a sequence of $\gamma_{n}\in \mathcal{S}_{2}(L(2,1),\lambda_{n})$ with $\mu(\gamma_{n})=1$ which converges to $\gamma$ in $C^{\infty}$. Note that $\mu(\gamma
        ^{2})=3$. Suppose that $\gamma$ is not simple, that is, there is a simple orbit $\gamma'$ and $k\in \mathbb{Z}_{>0}$ with $\gamma'^{k}=\gamma$. Form the lower semi-continuity of $\mu$, we have $\mu(\gamma_{n}^{2})\to\mu(\gamma'^{2k})=\mu((\gamma'^{2})^{k})=3$. Note that here we use the fact that $\gamma'^{2k}$ is contractible and $\mu((\gamma'^{2})^{k})\geq3$. But since $\gamma'^{2}$ is also contractible, we have $\mu(\gamma'^{2})\geq3$. Hence it follows from Proposition \ref{conleycovering} that $\mu((\gamma'^{2})^{k})\geq2k+1\geq5$. This is a contradiction. Therefore $\gamma$ is simple. This implies that for sufficiently large $n$, $\gamma_{n}$ is transversally isotopic to $\gamma$. This implies that $\gamma$ is 2-unknotted and  has self-linking number $-\frac{1}{2}$.
        
        At last, we prove $\mu(\gamma)=1$. Form the lower semi-continuity of $\mu$, we have $\mu(\gamma_{n})\to\mu(\gamma)=1$ or $0$. Suppose $\mu(\gamma)=0$. Then from Proposition \ref{conleycovering}, we have $\mu(\gamma^{2})=0$. This contradicts the assumption of dynamical convexity. Hence we have $\mu(\gamma)=1$. We complete the proof.
    \end{proof}
    As discussion so far, there is a sequence of $\gamma_{n}\in \mathcal{S}_{2}(L(2,1),f_{n}\lambda)$ with $\mu(\gamma_{n})=1$ and $\gamma\in \mathcal{S}_{2}(L(2,1),\lambda)$ with $\mu(\gamma)=1$ such that $\int_{\gamma_{n}}f_{n}\lambda\to \frac{1}{2}L$ and $\gamma_{n}$ converges to  $\gamma$ of $\lambda$  in $C^{\infty}$-topology. Therefore we have $\int_{\gamma}\lambda = \frac{1}{2}\,c_{1}^{\mathrm{ECH}}(L(2,1),\lambda)$ in $C^{\infty}$-topology. we complete the proof of Proposition \ref{degeneratel2}.
\end{proof}

Now, we have Proposition \ref{degeneratel2}. Therefore in order to complete the proof of Theorem \ref{maintheorem} (A), it is sufficient to show the next proposition.
\begin{prp}\label{upperbound}
     Assume that $(L(2,1),\lambda)$ is strictly convex. Then
    \begin{equation}
   \frac{1}{2}\,c_{1}^{\mathrm{ECH}}(L(2,1),\lambda) \leq \inf_{\gamma\in \mathcal{S}_{2},\mu(\gamma)=1}\int_{\gamma}\lambda .
\end{equation}
\end{prp}
\begin{proof}[\bf Proof of Proposition \ref{upperbound}]
    We prove this by contradiction. Suppose that there exists $\gamma_{\lambda}\in \mathcal{S}_{2}(L(2,1),\lambda)$ with $\mu(\gamma_{\lambda})=1$ such that $\frac{1}{2}\,c_{1}^{\mathrm{ECH}}(L(2,1),\lambda)>\int_{\gamma_{\lambda}}\lambda$.

    \begin{lem}\label{seqsmooth}
         There exists a sequence of smooth functions $f_{n}:L(2,1)\to \mathbb{R}_{>0}$  such that $f_{n}\to 1$ in $C^{\infty}$-topology and satisfying $f_{n}|_{\gamma_{\lambda}}=1$ and $df_{n}|_{\gamma_{\lambda}}=0$. Moreover, all periodic orbits of $X_{f_{n}\lambda}$ of periods $<n$ are non-degenerate and all contractible orbits of periods $<n$ have Conley-Zehnder index $\geq 3$. In addition, $\gamma_{\lambda}$ is a non-degenerate periodic orbit of $X_{f_{n}\lambda}$ with $\mu(\gamma_{\lambda})=1$ for every $n$.
    \end{lem}
    \begin{proof}[\bf Proof of Lemma \ref{seqsmooth}]
        See \cite[Lemma 6.8, 6.9]{HWZ4}
    \end{proof}

For a sequence of smooth functions $f_{n}:L(2,1)\to \mathbb{R}_{>0}$ in Lemma \ref{seqsmooth}, fix $N>>0$ sufficient large so that $c_{1}^{\mathrm{ECH}}(L(2,1),f_{N}\lambda)>\int_{\gamma_{\lambda}}\lambda$ and $N>2c_{1}^{\mathrm{ECH}}(L(2,1),f_{N}\lambda)$. We can take such $f_{N}$ because $c_{1}^{\mathrm{ECH}}$ is continuous in $C^{0}$-topology.

\begin{lem}\label{riginalchange}
    Let $f:L(2,1)\to \mathbb{R}_{>0}$ be a smooth function such that $f(x)<f_{N}(x)$ for any $x\in L(2,1)$. SUppose that $f\lambda$ is non-degenerate dynamically convex. Then there exists a simple periodic orbit $\gamma \in \mathcal{S}_{2}(L(2,1),f\lambda)$ 
with $\mu(\gamma)=1$ such that $\int_{\gamma}f\lambda<\int_{\gamma_{\lambda}}\lambda$.

\end{lem}
\begin{proof}[\bf Outline of the proof of 
Lemma \ref{riginalchange}]
See \cite[Proposition 3.1]{HrS}. In the proof and statement of \cite[Proposition 3.1]{HrS}, ellipsoids are used instead of $(L(2,1),f_{N}\lambda)$, but the important point in the proof is to find $2$-unknotted self-linking number $-\frac{1}{2}$ orbit with Conley-Zehnder index $1$ and construct a suitable $J$-holomorphic curve from \cite[Proposition 6.8]{HrLS}. Now, we have $\gamma_{\lambda} \in \mathcal{S}_{2}(L(2,1),f_{N}\lambda)$ 
with $\mu(\gamma_{\lambda})=1$ and hence  by applying \cite[Proposition 6.8]{HrLS}, we can construct a suitable $J$-holomorphic curve. By using this curves instead of ones in the original proof, we can show Proposition \ref{riginalchange}. Here we note that the discussion in the proof of Lemma \ref{cocycle} is needed to prove the same result of \cite[Theorem 3.15]{HrS}
\end{proof}
Now, we would complete the proof of Proposition.  Let $f:L(2,1)\to \mathbb{R}_{>0}$ be a smooth function such that $f(x)<f_{N}(x)$ for any $x\in L(2,1)$, $f\lambda$ be non-degenerate strictly convex and $\int_{\gamma_{\lambda}}\lambda<c_{1}^{\mathrm{ECH}}(L(2,1),f\lambda)<c_{1}^{\mathrm{ECH}}(L(2,1),f_{N}\lambda)$. We can check easily that it is possible to take such $f$.  Due to Lemma \ref{riginalchange}, there exists a simple periodic orbit $\gamma \in \mathcal{S}_{2}(L(2,1),f\lambda)$ 
with $\mu(\gamma)=1$ such that $\int_{\gamma}f\lambda<\int_{\gamma_{\lambda}}\lambda$. Since  $\inf_{\gamma\in \mathcal{S}_{2},\mu(\gamma)=1}\int_{\gamma}f\lambda = \frac{1}{2}\,c_{1}^{\mathrm{ECH}}(L(2,1),f\lambda)$, we have $\int_{\gamma_{\lambda}}\lambda<c_{1}^{\mathrm{ECH}}(L(2,1),f\lambda)\leq \int_{\gamma}f\lambda$. This is a contradiction. We complete the proof.
\end{proof}
\subsection{Case of $p=3,4,6$}
At first, we consider general $p$ and after that, we focus on $p=3,4,6$.

Suppose that $(L(p,p-1),\lambda)$ is strictly convex. Then we can take a sequence of strictly convex non-degenerate contact form with  $\lambda_{n}\to \lambda$.  Hence from Proposition \ref{prp1}, we can find $\gamma_{n}\in \mathcal{S}_{p}(L(p,p-1),\lambda_{n})$ with $\mu(\gamma_{n})=1$ such that $ \int_{\gamma_{n}}\lambda_{n}\leq \frac{1}{2}\,c_{1}^{\mathrm{ECH}}(L(p,p-1),\lambda_{n})$. 
If $\gamma_{n}$ converges to a simple orbit $\gamma$, It follows  from the same mathod of Proposition \ref{degeneratel2} and Claim \ref{limittingorbit} that $\gamma\in \mathcal{S}_{p}(L(p,p-1),\lambda)$ and $\mu(\gamma)=1$. But since $\mu(\gamma^{p})$ may be larger than $3$ for $p\geq 3$, the limiting orbit $\gamma$ of $\gamma_{n}$ may not be simple and hence we can't say $\gamma\in \mathcal{S}_{p}(L(p,p-1),\lambda)$ directly. In this subsection, we observe the limiting behavior of $\gamma_{n}$ and find a subsequence converging to a simple orbit $\gamma$ in the case of $p=3,4,6$.

At first, we recall how to find $\gamma_{n}\in \mathcal{S}_{p}$ in $(L(p,p-1),\lambda_{n})$. Due to Section 3,4, from the algenbraic structure of ECH and the behaviors of $J$-holomorphic curves, we can find an ECH generator $\alpha_{n}$ satisfying $\langle U_{J_{n},z}\alpha_{n},\emptyset \rangle \neq 0$ and $A(\alpha_{n})\leq c_{1}^{\mathrm{ECH}}(L(p,p-1),\lambda_{n})$. In addition, we can see that if $\langle U_{J_{n,z}}\alpha_{n},\emptyset \rangle \neq 0$, $\alpha_{n}$ is described as either
\begin{itemize}
    \item[(1).] $\alpha_{n}=(\gamma_{n},p)$ with $\gamma_{n}\in \mathcal{S}_{p}$, $\mu(\gamma_{n}^{p})=3$ and $\mu(\gamma_{n})=1$, or
    \item[(2).]$\alpha_{n}=(\gamma_{n,1},1)\cup{(\gamma_{n,2},1)}$ with $\gamma_{n,1}, \gamma_{n,2}\in \mathcal{S}_{p}$, $\mu(\gamma_{n,1})=\mu(\gamma_{n,2})=1$.
\end{itemize}
In any cases, we obtain exactly what we want.

Now, for $\lambda_{n}\to \lambda$, we pick and fix a sequence of ECH
generators $\alpha_{n}$ satisfying  $\langle U_{J_{n},z}\alpha_{n},\emptyset \rangle \neq 0$ and $A(\alpha_{n})\leq c_{1}^{\mathrm{ECH}}(L(p,p-1),\lambda_{n})$.

\begin{lem}\label{simpleconver}
    Suppose that $\{\alpha_{n}\}$ consists of $\alpha_{n}=(\gamma_{n},p)$ with $\gamma_{n}\in \mathcal{S}_{p}$, $\mu(\gamma_{n}^{p})=3$ and $\mu(\gamma_{n})=1$. Then there exist a simple orbit $\gamma$ of $(L(p,p-1),\lambda)$ and a subsequence $\{\gamma_{n_{k}}\}$ such that $\gamma_{n_{k}}\to \gamma$. In particular, $\gamma\in \mathcal{S}_{p}$, $\mu(\gamma^{p})=3$ and $\mu(\gamma)=1$.
\end{lem}
\begin{proof}[\bf Proof of Lemma \ref{simpleconver}]
    The proof is the same as the one of Claim \ref{limittingorbit}.
\end{proof}

Next, suppose that  $\{\alpha_{n}\}$ consists of $\alpha_{n}=(\gamma_{n,1},1)\cup{(\gamma_{n,2},1)}$ with $\gamma_{n,1}, \gamma_{n,2}\in \mathcal{S}_{p}$, $\mu(\gamma_{n,1})=\mu(\gamma_{n,2})=1$. By  Arzel`a–Ascoli theorem, we may assume that there are simple orbits $\gamma_{\infty,1}$ and $\gamma_{\infty,2}$ such that $\gamma_{n,1}\to \gamma_{\infty,1}^{k_{1}}$ and $\gamma_{n,2} \to \gamma_{\infty,2}^{k_{2}}$ for some $k_{1},k_{2}\in \mathbb{Z}_{>0}$.

\begin{lem}\label{hoplinks}
 If $\gamma_{\infty,1}\neq \gamma_{\infty,2}$, then $k_{1}=k_{2}=1$. In particular, $\gamma_{\infty,1},\gamma_{\infty,2}\in \mathcal{S}_{p}$ and $\mu(\gamma_{\infty,1})=\mu(\gamma_{\infty,2})=1$.
\end{lem}
\begin{proof}[\bf Proof of Lemma \ref{hoplinks}]
Without loss of generality, we focus on $i=2$. Note that $H_{1}(L(p,p-1)\backslash \gamma_{n,1})\cong \mathbb{Z}$ and $[\gamma_{n,2}]$ generates $H_{1}(L(p,p-1)\backslash \gamma_{n,1})$.
    Take small neighborhoods $\gamma_{\infty,i}\subset V_{i}$ for $i=1,2$ satisfying $V_{1}\cap{V_{2}}=\emptyset$. Then, $\gamma_{n,i}\subset V_{i}$ for sufficiently large $n$. This means that when we fix large $n_{0}$, $\gamma_{n,2}$ is isotopic to $\gamma_{\infty,2}^{k}$ in $L(p,p-1)\backslash \gamma_{n_{0},1}$ for any $n\geq n_{0}$. Hence  $[\gamma_{n,2}]$ generates $H_{1}(L(p,p-1)\backslash \gamma_{n_{0},1})$ and $[\gamma_{n,2}]=k_{2}[\gamma_{\infty,2}]$. Therefore we have $k_{2}=1$. In particular, by the same method with Claim \ref{limittingorbit}, we have $\gamma_{\infty,2}\in \mathcal{S}_{p}$ and $\mu(\gamma_{\infty,2})=1$.  We complete the proof.
\end{proof}

\begin{lem}\label{mutuallyprime}
  If $\gamma_{\infty,1}= \gamma_{\infty,2}$, then $k_{1}+k_{2}=p$. In addition, both $k_{1}$ and $k_{2}$ are mutually prime with $p$.
\end{lem}
\begin{proof}[\bf Proof of Lemma \ref{mutuallyprime}]

For simplicity, write $\gamma_{\infty}:=\gamma_{\infty,1}=\gamma_{\infty,2}$. Since $\gamma_{n,1}\to \gamma_{\infty}^{k_{1}}$ and $\gamma_{n,2} \to \gamma_{\infty}^{k_{2}}$, we have  $[\gamma_{n,1}]=k_{1}[\gamma_{\infty}]$ and $[\gamma_{n,2}]=k_{2}[\gamma_{\infty}]$ in $H_{1}(L(p,p-1))\cong \mathbb{Z}/p\mathbb{Z}$. Since $[\gamma_{n,1}]+[\gamma_{n,2}]=0$ and $[\gamma_{n,1}]$  generates $H_{1}(L(p,p-1))$, $[\gamma_{\infty}]$ also generates $H_{1}(L(p,p-1))$ and $k_{1}, k_{2}$ are mutually prime with $p$. Moreover, $k_{1}+k_{2}=kp$ for some $k\in \mathbb{Z}_{>0}$.

Suppose that $k\geq 2$. Then $k_{1}\geq p+1$ or $k_{2}\geq p+1$. Without loss of generality, we may assume $k_{1}\geq p+1$.  

 Since $\gamma_{\infty}^{p}$ is contractible, we have $\mu(\gamma_{\infty}^{p})\geq 3$ and hence $\mu((\gamma_{\infty}^{p})^{k_{1}})\geq 2k_{1}+1\geq2p+3$ (Proposition \ref{conleycovering}). 

On the other hand, for any $n$,  $2p-1\geq \mu(\gamma_{n,1}^{p})$. Indeed, if $\gamma_{n,1}$ is hyperbolic, $\mu(\gamma_{n,1}^{p})=p\mu(\gamma_{n,1})=p$. If $\gamma_{n,1}$ is elliptic, since  $\gamma_{n,1}$ is non-degenerate, there is $\theta_{n}\in \mathbb{R}\backslash \mathbb{Q}$ such that $\mu(\gamma_{n,1}^{m})=2\lfloor m\theta_{n} \rfloor+1$ for every $m\in \mathbb{Z}_{>0}$. Note that $0<\theta_{n}<1$ because $\mu(\gamma_{n,1})=1$. Hence we have $\mu(\gamma_{n,1}^{p})= 2\lfloor p\theta_{n} \rfloor+1\leq 2p-1$.

Now, we have $2p-1\geq \mu(\gamma_{n,1}^{p})$. Since $\gamma_{n,1}^{p}\to \gamma_{\infty}^{pk_{1}}$ and $\mu$ is lower semi-continuous, we have $2p-1\geq \mu(\gamma_{\infty}^{pk_{1}})=\mu((\gamma_{\infty}^{p})^{k_{1}})$, but this contradicts $\mu((\gamma_{\infty}^{p})^{k_{1}})\geq 2k_{1}+1\geq2p+3$.  Therefore, we have $k=1$ and complete the proof.
\end{proof}
\begin{proof}[\bf Complete the proof in the case of $p=3,4,6$]
For $\lambda_{n}\to \lambda$, consider a sequence of ECH generators $\alpha_{n}$ satisfying  $\langle U_{J_{n},z}\alpha_{n},\emptyset \rangle \neq 0$ and $A(\alpha_{n})\leq c_{1}^{\mathrm{ECH}}(L(p,p-1),\lambda_{n})$. 

If $\{\alpha_{n}\}$ contains an infinity subsequence consisting of  $\alpha_{n}=(\gamma_{n,1},p)$ with $\gamma_{n}\in \mathcal{S}_{p}$, $\mu(\gamma_{n}^{p})=3$, we can apply Lemma \ref{simpleconver} and hence we obtain $\gamma \in \mathcal{S}_{p}$ satisfying $\mu(\gamma)=1$ and $p\int_{\gamma}\lambda \leq c_{1}^{\mathrm{ECH}}(L(p,p-1),\lambda)$.

If $\{\alpha_{n}\}$ contains an infinity subsequence consisting of $\alpha_{n}=(\gamma_{n,1},1)\cup{(\gamma_{n,2},1)}$ with $\gamma_{n,1}, \gamma_{n,2}\in \mathcal{S}_{p}$, $\mu(\gamma_{n,1})=\mu(\gamma_{n,2})=1$. In addition, suppose that $\gamma_{n,1}$ and $\gamma_{n,2}$ converge to different orbits, then we can apply Lemma \ref{hoplinks} and hence we obtain two simple orbits $\gamma_{\infty,1} \gamma_{\infty,2} \in \mathcal{S}_{p}$ with $\mu(\gamma_{\infty,1})=\mu(\gamma_{\infty,2})=1$ satisfying $\gamma_{n,1}\to \gamma_{\infty,1}$, $\gamma_{n,2} \to \gamma_{\infty,2}$. Hence we have $ \int_{\gamma_{\infty,1}}\lambda + \int_{\gamma_{\infty,2}}\lambda\leq c_{1}^{\mathrm{ECH}}(L(p,p-1),\lambda)$.

If $\{\alpha_{n}\}$ contains an infinity subsequence consisting of $\alpha_{n}=(\gamma_{n,1},1)\cup{(\gamma_{n,2},1)}$ with $\gamma_{n,1}, \gamma_{n,2}\in \mathcal{S}_{p}$ and $\mu(\gamma_{n,1})=\mu(\gamma_{n,2})=1$. In addition, suppose that $\gamma_{n,1}$ and $\gamma_{n,2}$ converge to the same orbit $\gamma_{\infty}$ with some multiplicities $k_{1},k_{2}$. Then we can apply Lemma \ref{mutuallyprime}. If $p=3,4,6$, any pairs $(k_{1},k_{2})$ satisfying Lemma \ref{mutuallyprime} contain $1$. Therefore $(k_{1},k_{2})=(1,p-1)$ or $(p-1,1)$. This implies that $\gamma_{n,1}\to \gamma_{\infty}$, $\gamma_{n,2}\to \gamma_{\infty}^{p-1}$ or $\gamma_{n,1}\to \gamma_{\infty}^{p-1}$, $\gamma_{n,2}\to \gamma_{\infty}$ and hence we have $\gamma_{\infty}\in \mathcal{S}_{p}$ and $\mu(\gamma_{\infty})=1$ by the same way as Claim \ref{limittingorbit}. Moreover, we have $ p\int_{\gamma_{\infty}}\lambda\leq c_{1}^{\mathrm{ECH}}(L(p,p-1),\lambda)$.

In any cases, we complete the proof.
\end{proof}

\section{Construction of a family of $J$-holomorphic curves}
The purpose of this section is to prove Theorem \ref{familyofhol}.

Recall the conditions.

Let $\gamma_{1}:\mathbb{R}/T_{1}\mathbb{Z}\to Y$ and $\gamma_{2}:\mathbb{R}/T_{2}\mathbb{Z}\to Y$ be simple periodic orbits with $\mu(\gamma_{1})=\mu(\gamma_{2})=1$. Here we fix their parametrizations. Let $\mathcal{P}^{J}$ denote the set of $J$-holomorphic curves $(a,u):(\mathbb{R}\times S^{1},j_{0})\to (\mathbb{R}\times Y,J)$ with $u(s,t)\to\gamma_{1}(T_{1}t+e_{1})$ as $s\to +\infty$ and $u(s,t)\to\gamma_{1}(-T_{2}t+e_{2})$ as $s\to -\infty$ for some $e_{i}\in \mathbb{R}$ where $j_{0}(\partial_{s})=\partial_{t}$. Note that $a(s,t)\to +\infty$ as $s\to \pm \infty$. 

In order to prove the theorem, we recall the description of $P^{J}$ as a finite dimensional submanifold in a suitable infinite dimensional Banach manifold according to \cite{Dr}.  

For $\gamma_{i}$, there is an open neighborhood $\gamma_{i}(\mathbb{R}/T_{i}\mathbb{Z})\subset U_{i} \subset Y$, an open neighborhood $0 \in V_{i}\subset \mathbb{R}^{2}$ and a diffeomorphism $\phi_{i}:S^{1} \times V_{i}\to U_{i}$ satisfying  $\phi_{i}(t,0)=\gamma_{i}(T_{i}t)$ and $\phi_{i}^{*}\lambda=f_{i}\lambda_{0}$
where $f_{i}$ is a positive map $f_{i}:S^{1} \times V_{i} \to \mathbb{R}$ with $f_{i}(t,0,0)=T_{i}$, $df_{i}(t,0,0)=0$ and $\lambda_{0}$ denotes the standard contact structure on $\mathbb{R}^{3}$ \cite{HWZ1}. 

\begin{dfn}\cite[cf. Definiton 4]{Dr}
    Let $\delta>0$ and $q>2$. $(a,u)\in C^{\infty}(\mathbb{R}\times S^{1},\mathbb{R}\times Y)$ is $(\delta,1,q)$–convergence to $\gamma_{1}$, $-\gamma_{2}$ if it satisfies the following properties.
    \item[1.] there is a sufficient large number $R>>$ such that $u([R,+\infty)\times S^{1})\subset U_{1}$ and $u( (-\infty,-R]\times S^{1})\subset U_{2}$.
    \item[2.] Let $\phi^{-1}_{1}(u(s,t))=(\theta_{1}(s,t),z_{1}(s,t))$ and $\phi^{-1}_{2}(u(s,t))=(\theta_{2}(s,t),z_{2}(s,t))$. Then there are $d_{i}\in \mathbb{R}$ and $c_{i}\in \mathbb{R}$ such that 
\begin{equation}
       e^{\delta s}(a_{i}(\epsilon_{i} s,t)-\epsilon_{i} T_{i}s-d_{i}),\,\,e^{\delta s}(\theta_{i}(\epsilon_{i} s,t)-\epsilon_{i} t-c_{i})\in W^{1,q}([R,+\infty)\times S^{1},\mathbb{R})
   \end{equation}
   and 
\begin{equation}
    e^{\delta s}z_{i}(\epsilon_{i}s,t)\in W^{1,q}([R,+\infty)\times S^{1},\mathbb{R}^{2})
\end{equation}
where $\epsilon_{1}=+1$, $\epsilon_{2}=-1$.
\end{dfn}

Define $\mathcal{B}^{\infty}_{\delta} \subset   C^{\infty}(\mathbb{R}\times S^{1},\mathbb{R}\times Y)$ as the set consisting of $(\delta,1,q)$–convergence elements to $\gamma_{1}$, $-\gamma_{2}$ of $ C^{\infty}(\mathbb{R}\times S^{1},\mathbb{R}\times Y)$.

Note that for sufficiently small $\delta$, any element in $P^{J}$ is $(\delta,1,q)$–convergence and hence $P^{J} \subset \mathcal{B}^{\infty}_{\delta}$. See Proposition \ref{asymptics} or \cite[cf.Theorem 1.3]{HWZ1}.

Next we complete $\mathcal{B}^{\infty}_{\delta}$ to a suitable Banach manifold. Fix a Riemannian metric $g_{J}$ on $\mathbb{R}\times Y$ where
\begin{equation}
    g_{J}(a\partial_{s}+h,b\partial_{s}+k)=ab+\lambda(h)\lambda(k)+d\lambda(h,Jk).
\end{equation}

Let $(a,u)\in \mathcal{B}^{J}$ and $\Tilde{\gamma}:\mathbb{R}\times S^{1}\to T(\mathbb{R}\times Y)$ such that $\Tilde{\gamma}(s,t)\in T_{(a(s,t),u(s,t))}(\mathbb{R}\times Y)$.

\begin{dfn} \cite[cf. Definition 5.]{Dr}
For $\Tilde{\gamma}\in W^{1,q}_{\mathrm{loc}}((a,u)^{*}T(\mathbb{R}\times Y))$,
write 
\begin{equation*}
    \Tilde{\gamma}(s,t)=(b(s,t)\partial_{t},h(s,t)X_{\lambda}(u(s,t))+Q(u(s,t)))
\end{equation*}
where $Q(u(s,t))\in \xi_{u(s,t)}$.
If $b,h,Q$ satisfies, 
\begin{equation*}
    e^{\delta s}(b,h)\in W^{1,q}([R,\infty)]\times S^{1},\mathbb{R}^{2}),\,\,\, e^{-\delta s}(b,h)\in W^{1,q}((-\infty,-R]\times S^{1},\mathbb{R}^{2})
    \end{equation*}
    and 
     \begin{equation*}
     e^{\delta s} Q\in  W^{1,q}((u|_{[R,\infty)]\times S^{1}})^{*}\xi),\,\,e^{-\delta s} Q\in  W^{1,q}((u|_{(-\infty,-R]\times S^{1}})^{*}\xi)
\end{equation*}
for a sufficiently large $R>>0$. Then we say $\Tilde{\gamma}\in W^{1,q}_{\delta}((a,u)^{*}T(\mathbb{R}\times Y))$.
\end{dfn}

Let $\Tilde{h}=(a,u)\in \mathcal{B}^{\infty}$ and suppose that $\Tilde{h}([R, +\infty)\times S^{1})\subset U_{1}$, $\Tilde{h}( (-\infty,-R]\times S^{1})\subset U_{2}$ for a sufficient large $R>>0$. Let $\phi^{-1}_{1}(u(s,t))=(\theta_{1}(s,t),z_{1}(s,t))$ and $\phi^{-1}_{2}(u(s,t))=(\theta_{2}(s,t),z_{2}(s,t))$. Consider a smooth function $\kappa:\mathbb{R}\to [0,1]$ such that $\kappa(s)=0$ for $|s|<R+\frac{1}{2}$ and $\kappa(s)=1$ for $|s|>R+1$. For $d=(d_{1},d_{2}),c=(c_{1},c_{2})\in \mathbb{R}^{2}$,  define $\Tilde{h}_{(c,d)}$  as $\Tilde{h}_{(c,d)}=h$ on $[-R,R]\times S^{1}$ and 
\begin{equation}
\Tilde{h}_{(c,d)}(s,t)=(a(s,t)+\kappa(s)d_{i},\phi_{i}(\theta_{i}(s,t)+\kappa(s)c_{i},z_{i}(s,t)))
\end{equation}
on $\mathbb{R}\backslash [-R,R]\times S^{1}$

\begin{dfn}\cite[cf. Definition 6]{Dr}
    Fix small $\epsilon>0$ so that $2\epsilon$ is smaller than the injective radius with respect to $g_{J}$. We define
    \begin{equation}
        \mathcal{B}^{1,q}_{\delta}:=\{\mathrm{exp}_{\Tilde{h}_{(c,d)}}\circ \Tilde{\gamma}|\,\,\Tilde{\gamma}\in W^{1,q}_{\delta}(\Tilde{h}_{(c,d)}^{*}T(\mathbb{R}\times Y)),\,\, (c,d) \in \mathbb{R}^{4},\,\, |\Tilde{\gamma}|_{C^{0}}<\epsilon,\,\, |c_{i}|,\,|d_{i}|<\epsilon,\,i=1,2 \}
    \end{equation}
    where $\Tilde{h}\in \mathcal{B}^{\infty}_{\delta}$.
\end{dfn}

\begin{them}
    $\mathcal{B}^{1,q}_{\delta}$ is endowed with the differentiable structure of an infinite-dimensional, separable Banach manifold.
\end{them}
For a map $\Tilde{h}\in \mathcal{B}^{1,q}_{\delta}$, let 
\begin{equation}
    U=\{ (\Tilde{\gamma},(c,d)) \in  W^{1,q}_{\delta}(\Tilde{h}^{*}T(\mathbb{R}\times Y))\times \mathbb{R}^{4}|\,\,|\Tilde{\gamma}|_{C^{0}}<\epsilon,\,\, |c_{i}|,\,|d_{i}|<\epsilon,\,i=1,2 \}.
\end{equation}
Then by the construction, we can describe a local chart around $\Tilde{h}\in \mathcal{B}^{1,q}_{\delta}$ as 
\begin{equation}\label{localchart}
    E_{\Tilde{h}}:U\to \{\mathrm{exp}_{\Tilde{h}_{(c,d)}}\circ \Pi_{(c,d)}\Tilde{\gamma}|\,\,(\Tilde{\gamma},(c,d))\in U \}\subset \mathcal{B}^{1,q}_{\delta}
\end{equation}
where $E_{\Tilde{h}}(\Tilde{\gamma},(c,d))=\mathrm{exp}_{\Tilde{h}_{(c,d)}}\circ \Pi_{(c,d)}\Tilde{\gamma}$ and $\Pi_{(c,d)}:\Tilde{h}^{*}T(\mathbb{R}\times Y)\to \Tilde{h}_{(c,d)}^{*}T(\mathbb{R}\times Y)$  is the parallel transport along the shortest geodesic from a point of $\Tilde{h}$ to a point of $\Tilde{h}_{(c,d)}$. This implies that there is a natural identification
\begin{equation}
    T_{\Tilde{h}}\mathcal{B}^{1,q}_{\delta} \cong W^{1,q}_{\delta}(\Tilde{h}^{*}T(\mathbb{R}\times Y))\oplus
\mathbb{R}^{4}
\end{equation}
where $T_{\Tilde{h}}\mathcal{B}^{1,q}_{\delta}$ is the tangent space at $\Tilde{h}\in \mathcal{B}^{1,q}_{\delta}$.

Next, we consider the tangent space at $\Tilde{h}\in P^{J}$ as a subspace of $T_{\Tilde{h}}\mathcal{B}^{1,q}_{\delta}$.

From the standard argument, $T_{\Tilde{h}}P^{J}$ can be identified with the kernel of
\begin{equation}
    F_{\Tilde{h}}:T_{\Tilde{h}}\mathcal{B}^{1,q}_{\delta} \cong W^{1,q}_{\delta}(\Tilde{h}^{*}T(\mathbb{R}\times Y))\oplus
\mathbb{R}^{4}\to L_{\delta}^{q}(\wedge^{0,1}T^{*}(\mathbb{R}\times S^{1})\otimes \Tilde{h}^{*}T(\mathbb{R}\times Y)).
\end{equation}
Here, $F_{\Tilde{h}}$ is the linearization of the Cauchy-Riemann operator at $\Tilde{h}$ and described as follows.

Define 
\begin{equation}
    P_{\Tilde{h}}(\Tilde{\gamma},(c,d))=\Pi_{(c,d)}\circ \Phi_{\Tilde{h}}^{(c,d)}(\Pi_{(c,d)}\Tilde{\gamma})^{-1}\circ \Bar{\partial}_{J}\mathrm{exp}_{\Tilde{h}_{(c,d)}}\Pi_{(c,d)}\Tilde{\gamma}
\end{equation}
where $\Phi_{\Tilde{h}}^{(c,d)}(\Tilde{\xi}):T_{\Tilde{h}_{(c,d)}}(\mathbb{R}\times Y)\to T_{\mathrm{exp}_{\Tilde{h}_{(c,d)}}(\Tilde{\xi})}(\mathbb{R}\times Y)$ is the parallel transport for $\Tilde{\xi}\in T_{\Tilde{h}_{(c,d)}}(\mathbb{R}\times Y)$.

Then $F_{\Tilde{h}}$ is given by
\begin{equation}
    F_{\Tilde{h}}(\Tilde{\gamma},(c,d))=\frac{d}{d\lambda}P_{\Tilde{h}}(\lambda\Tilde{\gamma},\lambda(c,d))|_{\lambda=0}.
\end{equation}

Moreover, there is a natural identification
\begin{equation}
    T_{\Tilde{h}}P^{J}\cong \mathrm{Ker}F_{\Tilde{h}}
\end{equation}

From this construction, it can be described as
\begin{equation}
    F_{\Tilde{h}}(\Tilde{\gamma},(c,d))=D_{\Tilde{h}}(\Tilde{\gamma})+K(c,d)
\end{equation}
where $D_{\Tilde{h}}:W^{1,q}_{\delta}(\Tilde{h}^{*}T(\mathbb{R}\times Y))\to L_{\delta}^{q}(\wedge^{0,1}T^{*}(\mathbb{R}\times S^{1})\otimes_{J}\Tilde{h}^{*}T(\mathbb{R}\times Y))$ and $K:
\mathbb{R}^{4}\to L_{\delta}^{q}(\wedge^{0,1}T^{*}(\mathbb{R}\times S^{1})\otimes_{J} \Tilde{h}^{*}T(\mathbb{R}\times Y))$.

Note that $F_{\Tilde{h}}$ and $D_{\Tilde{h}}$ are Fredholm operators and we fix a generic $J$ so that $F_{\Tilde{h}}$ is surjective for any $\Tilde{h}\in P^{J}$. 

\begin{prp}\label{fredind}\cite[Theorem 9]{Dr}
    \begin{equation}
        \mathrm{Ind}F_{\Tilde{h}}=\mathrm{Ind}D_{\Tilde{h}}+4=4.
    \end{equation}

    In particular, $\mathrm{Ind}D_{\Tilde{h}}=0$.
\end{prp}

From Proposition \ref{fredind}, we can see that $\mathrm{pr}:T_{\Tilde{h}}P^{J}\cong \mathrm{Ker}F_{\Tilde{h}}\to \mathbb{R}^{4}$ is an isomorphism where $\mathrm{pr}$ is 
the restriction of the natural projection $\mathrm{pr}_{2}:T_{\Tilde{h}}\mathcal{B}^{1,q}_{\delta} \cong W^{1,q}_{\delta}(\Tilde{h}^{*}T(\mathbb{R}\times Y))\oplus
\mathbb{R}^{4}\to \mathbb{R}^{4}$ to $T_{\Tilde{h}}P^{J}\cong \mathrm{Ker}F_{\Tilde{h}}$.

Recall the moduli space $\mathcal{M}^{J}(\alpha,\emptyset)/\mathbb{R}\cong S^{1}$. This is a quotient space of $P^{J}$. Consider the projection $\pi:P^{J}\to  \mathcal{M}^{J}(\alpha,\emptyset)/\mathbb{R}\cong S^{1}$.  For $\Tilde{h}=(a,h)\in P^{J}$, 

\begin{equation}\label{fiber}
    \pi(\pi(\Tilde{h}))^{-1}=\{(a(s+d,t+c)+e,h(s+d,t+c))\in P^{J}|\,(c,d,e)\in S^{1}\times \mathbb{R}\times \mathbb{R}\,\}
\end{equation}

So $\mathcal{M}^{J}(\alpha,\emptyset)/\mathbb{R}\cong S^{1}$ is a fiber bundle whose fiber is isomorphic to $S^{1}\times \mathbb{R}\times \mathbb{R}$.

Take a section  $\Tilde{v}:\mathcal{M}^{J}(\alpha,\emptyset)/\mathbb{R}\cong S^{1} \to P^{J} $. 
For $\tau \in S^{1}$, write $\Tilde{v}(\tau)=(a_{\tau},v_{\tau})\in P^{J}$.
Let $p_{1}$, $p_{2}: S^{1}\to S^{1}$ denote the functions defined by
\begin{equation}
    \lim_{s\to +\infty}v_{\tau}(s,t)=\gamma_{1}(T_{1}(t+p_{1}(\tau))
\end{equation}
\begin{equation}
    \lim_{s\to -\infty}v_{\tau}(s,t)=\gamma_{2}(-T_{2}(t+p_{2}(\tau)))
\end{equation}

\begin{prp}\label{diffdouble}
    $p_{1}$, $p_{2}: S^{1}\to S^{1}$ are smooth functions such that $\iota=p_{1}-p_{2}:S^{1}\to S^{1}$ is a locally diffeomorphism.
\end{prp}
\begin{proof}[\bf Proof of Proposition \ref{diffdouble}]

In order to prove Proposition \ref{diffdouble}, we have to observe the tangent space along the fiber.

Fix $\tau_{0}\in S^{1}$. Take $e_{i}\in S^{1}$ and $f_{i}\in \mathbb{R}$ so that
\begin{equation}
    \lim_{\epsilon_{i}s\to +\infty}v_{\tau_{0}}(s,t)=\gamma_{i}(\epsilon_{i}T_{i}(t+e_{i})),\,\,  \lim_{\epsilon_{i}s\to +\infty}a_{\tau_{0}}(s,t)-\epsilon_{i}T_{i}s=f_{i}.
\end{equation}

Let $\Tilde{w}:(-\epsilon,\epsilon)\to P^{J}$ be a smooth path with $\Tilde{w}(0)=\Tilde{v}(\tau_{0})$ where $\epsilon>0$ is sufficiently small. Write $\Tilde{w}(\lambda)=(b_{\lambda},w_{\lambda})\in P^{J}$ and let $c_{i}:(-\epsilon,\epsilon)\to S^{1}$, $d_{i}:(-\epsilon,\epsilon)\to \mathbb{R}$ $i=1,2$ denote the functions defined by
\begin{equation}
    \lim_{\epsilon_{i}s\to +\infty}w_{\lambda}(s,t)=\gamma_{i}(\epsilon_{i}T_{i}(t+c_{i}(\lambda))),\,\,  \lim_{\epsilon_{i}s\to +\infty}b_{\lambda}(s,t)-\epsilon_{i}T_{i}s=d_{i}(\lambda).
\end{equation}

From the construction we can easily check that
\begin{equation}
    \mathrm{pr}_{2}\circ E^{-1}_{\Tilde{v}(\tau_{0})}(\Tilde{w}(\lambda))=(c_{1}(\lambda)-e_{1},c_{2}(\lambda)-e_{2},d_{1}(\lambda)-f_{1},d_{2}(\lambda)-f_{2})
\end{equation}
where $E_{\Tilde{v}_{\tau_{0}}}$ is the map defined by (\ref{localchart}).

This implies that for $\Tilde{w}'(0)\in T_{\Tilde{v}(\tau_{0})}P^{J}\cong \mathrm{Ker}F_{\Tilde{v}(\tau_{0})}$,
\begin{equation}
    \mathrm{pr}(\Tilde{w}'(0))=(c'_{1}(0),c'_{2}(0),d'_{1}(0),d'_{2}(0)).
\end{equation}

Under these understandings, we observe the tangent space along
the fiber at $\Tilde{v}(\tau_{0})$. For $i=1,2,3$, define $\Tilde{w}_{i}:(-\epsilon,\epsilon)\to P^{J}$ as 
\begin{equation}
    \Tilde{w}_{1}(\lambda)=(a_{\tau_{0}}(s,t)+\lambda,v_{\tau_{0}}(s,t)),
\end{equation}
\begin{equation}
    \Tilde{w}_{2}(\lambda)=(a_{\tau_{0}}(s+\lambda,t),v_{\tau_{0}}(s+\lambda,t)),
\end{equation}
\begin{equation}
    \Tilde{w}_{3}(\lambda)=(a_{\tau_{0}}(s,t+\lambda),v_{\tau_{0}}(s,t+\lambda)).
\end{equation}

Since (\ref{fiber}), the tangent space along
the fiber at $\Tilde{v}(\tau_{0})$ is spaned by $\Tilde{w}_{1}'(0)$, $\Tilde{w}_{2}'(0)$ and $\Tilde{w}_{3}'(0)$.
From the definition, we have
\begin{equation}
    \mathrm{pr}(\Tilde{w}_{1}'(0))=(0,0,1,1),\,\,\mathrm{pr}(\Tilde{w}_{2}'(0))=(0,0,T_{1},-T_{2}),\,\,\mathrm{pr}(\Tilde{w}_{3}'(0))=(1,1,0,0).
\end{equation}

Since $\Tilde{v}:\mathcal{M}^{J}(\alpha,\emptyset)/\mathbb{R}\cong S^{1} \to P^{J} $ is a section, $\Tilde{v}'(\tau_{0})$, $\Tilde{w}_{1}'(0)$, $\Tilde{w}_{2}'(0)$ and $\Tilde{w}_{3}'(0)$ span $T_{\Tilde{v}(\tau_{0})}P^{J}\cong \mathrm{Ker}F_{\Tilde{v}(\tau_{0})}$. 
From the definition, the first two coordinate of  $\mathrm{pr}(\Tilde{v}'(\tau_{0}))$ is $(p_{1}'(\tau_{0}),p_{2}'(\tau_{0}))$. Since $\mathrm{pr}:T_{\Tilde{h}}P^{J}\cong \mathrm{Ker}F_{\Tilde{h}}\to \mathbb{R}^{4}$ is an isomorphism, $\mathrm{pr}(\Tilde{w}_{1}'(0))$, $\mathrm{pr}(\Tilde{w}_{2}'(0))$, $\mathrm{pr}(\Tilde{w}_{3}'(0))$ and  $\mathrm{pr}(\Tilde{v}'(\tau_{0}))$ span $\mathbb{R}^{4}$ and so we have $p_{1}'(\tau_{0})-p_{2}'(\tau_{0}) \neq 0$. This implies that $\iota=p_{1}-p_{2}:S^{1}\to S^{1}$ is a local diffeomorphism and hence we complete the proof.
\end{proof}

\begin{prp}\label{doublecove}
    $\iota=p_{1}-p_{2}:S^{1} \to S^{1}$ is  a p-fold cover.
\end{prp}
\begin{proof}[\bf Proof of Proposition \ref{doublecove}]
By considering a composition function of $t \to t-p_{2}(\tau)$ with $\Tilde{v}(\tau)$, we may assume that for $\Tilde{v}(\tau)=(a_{\tau},v_{\tau})\in P^{J}$,
\begin{equation}\label{plimit}
    \lim_{s\to +\infty}v_{\tau}(s,t)=\gamma_{1}(T_{1}(t+p(\tau)))
\end{equation}
\begin{equation}
    \lim_{s\to -\infty}v_{\tau}(s,t)=\gamma_{2}(-T_{2}t)
\end{equation}
Moreover by considering $\tau \to -\tau$, we may assume that $\iota:S^{1}\to S^{1}$ is an orientation preserving map.

For $R>>0$ Consider two solid torus $\bigcup_{\tau\in S^{1}}v_{\tau}((-\infty,-R]\times S^{1})\cup \mathrm{Im}\gamma_{2}$ and $\bigcup_{\tau\in S^{1}}v_{\tau}([R,+\infty)\times S^{1})\cup \mathrm{Im}\gamma_{1}$ whose boundaries have coordinates 
\begin{equation}
S^{1}\times S^{1} \ni (\tau,t) \mapsto v_{\tau}(-R,t),\, v_{\tau}(R,t).
\end{equation}
From the construction, we can see that the map on the boundaries
\begin{equation}
    \psi:\bigcup_{\tau\in S^{1}}v_{\tau}(\{-R\}\times S^{1})\to  \bigcup_{\tau\in S^{1}}v_{\tau}(\{R\}\times S^{1}),\,\,v_{\tau}(-R,t) \mapsto v_{\tau}(R,t)
\end{equation}
 is a diffeomorphism and  the lens space constructed by gluing the two solid torus by this map is diffeomorphic to $L(p,p-1)$. This implies that for fixed $\tau_{0} \in S^{1}$, $\psi_{*}([\bigcup_{\tau\in S^{1}}v_{\tau}(-R,0)])=p [v_{\tau_{0}}(\{-R\}\times S^{1})]$ where 
 \begin{equation}
     \psi_{*}:H_{1}(\bigcup_{\tau\in S^{1}}v_{\tau}(\{-R\}\times S^{1}))\to H_{1} (\bigcup_{\tau\in S^{1}}v_{\tau}(\{R\}\times S^{1}))
 \end{equation}
 is the  induced map by $\psi$. Consider the  inclusion $i:\bigcup_{\tau\in S^{1}}v_{\tau}(\{R\}\times S^{1}) \to \bigcup_{\tau\in S^{1}}v_{\tau}([R,+\infty)\times S^{1})\cup \mathrm{Im}\gamma_{1}$. Then $i_{*}\circ \psi_{*}([\bigcup_{\tau\in S^{1}}v_{\tau}(-R,0)])= p [\gamma_{1}]$.   This implies that the multiplicity of $\iota:S^{1}\to S^{1}$ is $p$. We complete the proof of Proposition \ref{doublecove}.
 \end{proof}

 \begin{proof}[\bf Proof of Theorem \ref{familyofhol}]
Since $\iota:S^{1}\to S^{1}$ is a $p$-fold cover,  by considering a composite function $\tau \mapsto \tau-\tau_{0}$, we may assume that $\iota(0)=0$. Let $\Tilde{S}^{1}=\mathbb{R}/p\mathbb{Z}$ and $\pi:\Tilde{S}^{1}\to S^{1}$ be the natural projection. Then there is a lift $\Tilde{p}:(S^{1},0)\to (\Tilde{S}^{1},0)$. Since  $\Tilde{p}$ is a diffeomorphism, the map $\Tilde{p}^{-1}:\Tilde{S}^{1} \to S^{1}$ is defined and $p \circ \Tilde{p}^{-1}=\pi$. Define $j:S^{1}\to S^{1}$ as $j(\tau)=p\tau$. then there is a lift $\Tilde{j}:(S^{1},0)\to (\Tilde{S}^{1},0)$ with $\pi \circ \Tilde{j} =j$. From the construction, we have $p \circ \Tilde{p}^{-1} \circ \Tilde{j}=j$.

Define $\Tilde{u}(\tau)=\Tilde{v}(\Tilde{p}^{-1} \circ \Tilde{j}(\tau))$. Then this is exactly what we want. Indeed 
\begin{equation}\label{plimit}
    \lim_{s\to +\infty}u_{\tau}(s,t)=\gamma_{1}(T_{1}(t+p\circ \Tilde{p}^{-1} \circ \Tilde{j}(\tau))=\gamma_{1}(T_{1}(t+j(\tau)))=\gamma_{1}(T_{1}(t+p\tau))
\end{equation}
\begin{equation}
    \lim_{s\to -\infty}u_{\tau}(s,t)=\gamma_{2}(-T_{2}t).
\end{equation}

We complete the proof of Theorem \ref{familyofhol}.
 \end{proof}


\begin{thebibliography}{99}
\bibitem[BE]{BE}K. Baker and J. Etnyre,
Rational linking and contact geometry. Perspectives in analysis, geometry, and topology, 19–37, Progr. Math., 296, Birkhäuser/Springer, New York, 2012.


\bibitem[CHP]{CHP} D. Cristofaro-Gardiner, M. Hutchings and D. Pomerleano, Torsion contact forms in three dimensions have two or infinitely many Reeb orbits,  Geom. Top, 23 (2019), 3601-3645.
\bibitem[CHR]{CHR}D. Cristofaro-Gardiner, M. Hutchings and V.G.B. Ramos, The asymptotics of ECH capacities, Invent. Math. 199 (2015), 187-214
\bibitem[Dr]{Dr} D. Dragnev, Fredholm theory and transversality for noncompact pseudoholomorphic maps in symplectizations, Comm. Pure Appl. Math 57 (2004), 726–763.
\bibitem[Fe]{Fe} B. Ferreira, Elliptic Reeb orbit on some real projective three-spaces via ECH, 	arXiv:2305.06257 
\bibitem[HWZ1]{HWZ1}H. Hofer, K. Wysocki, and E. Zehnder, Properties of pseudoholomorphic curves in symplectisations. I. Asymptotics. Ann. Inst. H. Poincaré C Anal. Non Linéaire 13 (1996), no. 3, 337–379.
\bibitem[HWZ2]{HWZ2}  H. Hofer, K. Wysocki, and E. Zehnder, Properties of pseudo-holomorphic curves in symplectisations. II. Embedding controls and algebraic invariants. Geom. Funct. Anal. 5 (1995), no. 2, 270–328.
\bibitem[HWZ3]{HWZ3} H. Hofer, K. Wysocki, and E. Zehnder, A characterisation of the tight three-sphere. Duke J.
Math, 81(1):159–226, 1996.
\bibitem[HWZ4]{HWZ4}H. Hofer, K. Wysocki, and E. Zehnder, The dynamics on three-dimensional strictly convex energy surfaces. Ann. of Math. (2) 148 (1998), no. 1, 197–289.
\bibitem[HWZ5]{HWZ5} H. Hofer, K. Wysocki, and E. Zehnder, A characterization of the tight 3-sphere. II. Comm. Pure Appl. Math. 52 (1999), no. 9, 1139–1177.
\bibitem[Hr1]{Hr1} U. L. Hryniewicz, Fast finite-energy planes in symplectizations and applications. Transactionsof the American Mathematical Society, 364(4):1859–1931, 2012.
\bibitem[Hr2]{Hr2}U. L. Hryniewicz, Systems of global surfaces of section for dynamically convex Reeb flows on the 3-sphere.  J. Symplectic Geom. 12 (2014), no. 4, 791–862.

\bibitem[HrHuRa]{HrHuRa} U. L. Hryniewicz,  M. Hutchings and V.G.B. Ramos, Unknotted Reeb orbits and the first ECH capacity, in preparation.
\bibitem[HrLS]{HrLS} U. L. Hryniewicz, J. E. Licata, and P. A. S. Salomão, A dynamical characterization of universally tight lens spaces. Proceedings of the London Mathematical Society, 110(1):213–269, 2015
\bibitem[HrS]{HrS} U. L. Hryniewicz,  and P. A. S. Salomão, Elliptic bindings for dynamically convex Reeb flows on the real projective three-space. Calc. Var. Partial Differential Equations 55 (2016), no. 2, Art. 43, 57 pp.
\bibitem[H1]{H1}M. Hutchings, An index inequality for embedded pseudoholomorphic curves in symplectizations,
J. Eur. Math. Soc. 4 (2002), 313-361.
\bibitem[H2]{H2}M. Hutchings, Quantitative embedded contact homology J. Diff. Geom. 88 (2011), 231-266.
\bibitem[H3]{H3}M. Hutchings, Lecture notes on embedded contact homology. Contact and symplectic topology, 389–484, Bolyai Soc. Math. Stud., 26, János Bolyai Math. Soc., Budapest, 2014.
\bibitem[HT1]{HT1}M. Hutchings and C. H. Taubes, Gluing pseudoholomorphic curves along branched covered cylinders I, 
J. Symp. Geom. 5 (2007), 43-137.
\bibitem[HT2]{HT2}M. Hutchings and C. H. Taubes, Gluing pseudoholomorphic curves along branched covered cylinders II,
J. Symp. Geom. 7 (2009), 29-133.
\bibitem[HT3]{HT3}M. Hutchings and C. H. Taubes, The Weinstein conjecture for stable Hamiltonian structures,  Geom. Top, 13 (2009), 901-941.
\bibitem[KM]{KM}P. Kronheimer and T. Mrowka, Monopoles and three-manifolds, New Math. Monogr. 10, Cambridge University Press (2007).
\bibitem[Sch]{Sch}A. Schneider,
Global surfaces of section for dynamically convex Reeb flows on lens spaces. Trans. Amer. Math. Soc. 373 (2020), no. 4, 2775–2803.
\bibitem[T]{T1}C. H. Taubes, Embedded contact homology and Seiberg-Witten Floer homology I, Geom. Top. 14 (2010), 2497–2581.




















\end{thebibliography}
\end{document}